\documentclass[10pt,reqno]{amsart}
\usepackage{amsmath,amssymb,amsthm,graphicx,epstopdf,mathrsfs,url}
\usepackage[usenames,dvipsnames]{color}
\usepackage{dsfont}   
\definecolor{darkred}{rgb}{0.7,0.1,0.1}
\definecolor{darkblue}{rgb}{0.1,0.1,0.4}
\definecolor{darkgrey}{rgb}{0.5,0.5,0.5}
\usepackage[colorlinks=true,linkcolor=darkred,citecolor=Blue]{hyperref}

\numberwithin{equation}{section}
\theoremstyle{plain}
\newtheorem{thm}{Theorem}[section]
\newtheorem{lem}[thm]{Lemma}

\newtheorem{prop}[thm]{Proposition}
\newtheorem{cor}[thm]{Corollary}

\newtheorem{definition}[thm]{Definition}
\theoremstyle{remark}
\newtheorem{remark}[thm]{Remark}

\theoremstyle{plain}
\newtheorem{hypothesis}[thm]{Hypothesis}

\newcommand{\hyp}[1]{$C^{2}$-hypersurface as in Definition~\ref{definition_hypersurface}}

\newcommand{\sign}{\mathrm{sign}\,}   

\DeclareMathOperator\ran{ran}

\newcommand{\dom}{\mathrm{dom}\,}

\begin{document}
\title[]{ On two-dimensional Dirac operators with $\delta$-shell interactions supported on unbounded curves with straight ends}


\author[J. Behrndt]{Jussi Behrndt}
\address{Institut f\"{u}r Angewandte Mathematik\\
Technische Universit\"{a}t Graz\\
 Steyrergasse 30, 8010 Graz, Austria\\
E-mail: {\tt behrndt@tugraz.at}}

\author[P. Exner]{Pavel Exner}
\address{Doppler Institute for Mathematical Physics and Applied Mathematics\\ 
Czech Technical University in Prague\\ B\v{r}ehov\'{a} 7, 11519 Prague, Czech Republic,
{\rm and}
Department of Theoretical Physics\\
Nuclear Physics Institute, Czech Academy of Sciences, 
25068 \v{R}e\v{z}, Czech Republic\\
E-mail: {\tt exner@ujf.cas.cz}
}

\author[M. Holzmann]{Markus Holzmann}
\address{Institut f\"{u}r Angewandte Mathematik\\
Technische Universit\"{a}t Graz\\
 Steyrergasse 30, 8010 Graz, Austria\\
E-mail: {\tt holzmann@math.tugraz.at}}

\author[M. Tu\v{s}ek]{Mat\v{e}j Tu\v{s}ek}
\address{Department of Mathematics, Faculty of Nuclear Sciences and Physical Engineering\\
Czech Technical University in Prague \\
Trojanova 13, 120 00, Prague \\
E-mail: {\tt matej.tusek@fjfi.cvut.cz}}


\maketitle

\begin{abstract}
  In this paper we study the self-adjointness and spectral properties of two-dimensional Dirac operators with electrostatic, Lorentz scalar, and anomalous magnetic $\delta$-shell interactions with constant weights that are supported on a smooth unbounded curve that is straight outside a compact set and whose ends are rays that are not parallel to each other. For all possible combinations of interaction strengths we describe the self-adjoint realizations and compute their essential spectra. Moreover, we prove in different situations the existence of geometrically induced discrete eigenvalues.
\end{abstract}

\section{Introduction}

Operators with contact interactions are frequently used as quantum mechanical Hamiltonians. True, as models of real physical systems they are idealized but often they provide a description which preserves the salient features of the system one wants to describe while being mathematically more accessible than their `regular' counterparts. In nonrelativistic quantum mechanics there is wealth of such models; for a survey and bibliography we refer to the monograph \cite{AGHH05}. The motivation to investigate Dirac operators with contact interactions was weaker for a long time; the only system of that type appeared in connection with the so-called MIT bag model \cite{CJJTW74}. This changed with the advent of graphene and similar materials, in which the effective dynamics is described by the Dirac equation, see, e.g., \cite{Be08}, and the systems of interest may have many geometric features depending on their experimental design.

The goal in this paper is to study the self-adjointness and the spectrum of a Dirac operator in $\mathbb{R}^2$ with a combination of electrostatic, Lorentz scalar, and anomalous magnetic $\delta$-shell potentials supported on an unbounded curve $\Gamma$ that is straight outside a compact set and such that its ends are not parallel to each other. Such an operator is associated with the formal differential expression
\begin{equation} \label{eq:formal_expr}
  \mathcal{A}_{\eta, \tau, \lambda}^{m, c} := -i c (\sigma \cdot \nabla) + m c^2 \sigma_3 + (\eta \sigma_0 + \tau \sigma_3 + \lambda i (\sigma_1 \nu_1 + \sigma_2 \nu_2) \sigma_3) \delta_\Gamma,
\end{equation}
where $c>0$ is the speed of light, $m\geq 0$ is the mass of the particle (note that we will later consider $m\in\mathbb{R}$), $\delta_\Gamma$ stands for the $\delta$- or single layer distribution supported on $\Gamma$ and $\nu = (\nu_1, \nu_2)$ is a unit normal vector field along $\Gamma$; finally, $\eta,\tau,\lambda\in\mathbb R$ are coupling constants for the electrostatic, Lorentz scalar, and anomalous magnetic interaction, respectively. It is known that with these three parameters one can define a general symmetric $\delta$-interaction with constant coefficients for the two-dimensional Dirac operator; a fourth possible parameter in the $2 \times 2$-matrix-valued coefficient of $\delta_\Gamma$ can be gauged away by a unitary transformation, see  \cite{BHT23, CLMT21}.

Operators of the form $\mathcal{A}_{\eta, \tau, \lambda}^{m, c}$ were first studied in the one-dimensional case \cite{GS87}, in which the singular interaction is supported on a point in $\mathbb{R}$. By using this analysis and a decomposition to spherical harmonics, the three-dimensional counterpart of the operator in~\eqref{eq:formal_expr} was investigated in the end of the 1980s with the interaction support being a sphere \cite{DES89}. It took then 25 years, until more general bounded $C^2$-surfaces in $\mathbb{R}^3$ were considered as supports of purely electrostatic interactions in \cite{AMV14}. After that the interest on the operators induced by \eqref{eq:formal_expr} grew, the more general interactions including Lorentz scalar and anomalous magnetic terms were introduced, and the self-adjointness, the spectral and scattering properties were studied for bounded interaction supports in dimensions two and three, see \cite{AMV15, BEHL18, BEHL19_1, BHOP20, CLMT21}, to mention just a few of references, and the recent review \cite{BHSS22} for more references.  The case of unbounded interaction supports was not studied that extensively, despite the fact that for the case of the straight line in $\mathbb{R}^2$ interesting spectral transitions were observed in \cite{BHT22, BHT23}. We also mention the paper \cite{Ben21}, where the self-adjointness and spectral properties for combinations of electrostatic and Lorentz scalar interactions supported on surfaces in $\mathbb{R}^3$, which are compact perturbations of the plane, were investigated, the works \cite{FL23, PvdB22, R21, R22} about the self-adjointness of the operator in~\eqref{eq:formal_expr} under different assumptions on the interaction support including non-smooth broken lines, and the recent paper \cite{FHL23}, where the spectrum of the operator $\mathcal{A}_{0, \tau, 0}^{m, c}$ was studied, when $\Gamma$ is a broken line.

It is the main goal in this paper to treat the two-dimensional Dirac operator in~\eqref{eq:formal_expr} with a singular potential that is supported on a smooth unbounded curve in $\mathbb{R}^2$ that is straight outside a compact set in a systematic way. As for bounded interaction supports, the operator associated with $\mathcal{A}_{\eta, \tau, \lambda}^{m, c}$ lacks semi-boundedness and hence, form methods are not suitable to introduce it as a self-adjoint operator. Therefore, a different technique from extension theory of symmetric operators will be applied. We follow closely the approach in \cite{BHSS22}, where two- and three-dimensional Dirac operators with singular potentials supported on bounded interaction supports were studied with the help of boundary triples and their Weyl functions. More precisely, we construct a suitable ordinary boundary triple, which allows us to introduce $\mathcal{A}_{\eta, \tau, \lambda}^{m, c}$ as a self-adjoint operator for all combinations of the interaction strengths $\eta, \tau, \lambda\in\mathbb R$ and to identify the essential spectrum. In a similar way as it is known for bounded interaction supports, it turns out that there are special combinations of the interaction strengths $\eta, \tau, \lambda$ for which there is a lack of Sobolev regularity in the operator domain. Moreover, as in the case of the straight line \cite{BHT22, BHT23}, unexpected transitions in the essential spectra appear in several situations.

Finally, we are going to show for particular geometries of the interaction support and/or parameter values the existence of isolated eigenvalues. Relations between spectral properties of differential operators and the geometry belong to the most traditional questions, and one of the intriguing examples is the fact that geometrically induced bound states may exist even in regions which are spatially infinitely extended. A classical example is the Dirichlet Laplacian in strips around a bent reference curve, for which we refer to the monograph \cite{EK15} and the bibliography therein. The same effect was observed for two-dimensional Schr\"odinger operators with attractive
$\delta$-potentials supported on unbounded curves \cite{EI01, P17}, and recently also for Dirac operators
with Lorentz scalar $\delta$-shell interactions on broken lines \cite{FHL23}.
We will prove that also for smooth unbounded interaction supports $\Gamma$ that are straight outside a compact set geometrically induced eigenvalues may exist. We do this first 
in the case, where the electrostatic and the Lorentz scalar interaction strengths have the same negative value, no magnetic term is present, and the speed of light $c$ is large by using the nonrelativistic limit of the Dirac operator in combination with the known result about the corresponding Schrödinger operator from \cite{EI01}. Secondly, we prove the existence of geometrically induced bound states for attractive purely Lorentz scalar interactions that are supported on a compact perturbation of a broken line with a sufficiently small opening angle by using the claim proved in the paper \cite{FHL23}. While such existence results are of obvious interest, they lack the universal character of their nonrelativistic counterparts mentioned above, and their extension to a wider class of operators represents an important challenge.

The paper is organized as follows: In Section~\ref{section_main_results} we state and discuss the main results of this paper in detail. Then, in Section~\ref{section_preliminaries} we provide some preliminary material that is necessary to prove the main results in Sections~\ref{section_proof} and~\ref{section_nonrelativistic_limit}. Finally, in Appendix~\ref{appendix_bt} we recall some useful abstract results about ordinary boundary triples, and in Appendix~\ref{appendix_int_op} we prove a technical statement about compactness properties of integral operators.

\subsection*{Notations}

Let 
\begin{equation} \label{def_Pauli_matrices}
   \sigma_1 := \begin{pmatrix} 0 & 1 \\ 1 & 0 \end{pmatrix}, \qquad
   \sigma_2 := \begin{pmatrix} 0 & -i \\ i & 0 \end{pmatrix}, \qquad
   \sigma_3 := \begin{pmatrix} 1 & 0 \\ 0 & -1 \end{pmatrix},
\end{equation}
be the Pauli spin matrices and denote by $\sigma_0$ the $2 \times 2$-identity matrix. Note that the Pauli matrices satisfy 
\begin{equation} \label{anti_commutation}
  \sigma_j \sigma_k + \sigma_k \sigma_j = 2 \delta_{jk} \sigma_0, \quad j,k \in \{ 1,2,3\}.
\end{equation}
For $x = (x_1,x_2)$ we will often write $\sigma \cdot x = \sigma_1 x_1 + \sigma_2 x_2$ and $\sigma \cdot \nabla = \sigma_1 \partial_1 + \sigma_2 \partial_2$.

Let $\Omega$ be an open set.
The $L^2$-based Sobolev spaces of $k$ times weakly differentiable functions are denoted by $H^k(\Omega; \mathbb{C})$ and we write $H^k(\Omega; \mathbb{C}^2) := H^k(\Omega; \mathbb{C}) \otimes \mathbb{C}^2$. Moreover, if $\Gamma$ is an infinite curve that is straight outside a compact set as in~\eqref{arc_length_Gamma_intro} with arc-length parametrization $\gamma: \mathbb{R} \rightarrow \mathbb{R}^2$,
then we define the mapping
\begin{equation} \label{def_U}
  U_\gamma: \mathcal{S}'(\Gamma; \mathbb{C}) \rightarrow \mathcal{S}'(\mathbb{R}; \mathbb{C}), \quad (U_\gamma f)(\varphi) = f(\varphi( \gamma^{-1}(\cdot))).
\end{equation}
Clearly, if $f \in L^1(\Gamma; \mathbb{C})$, then the distribution $U_\gamma f \in L^1(\mathbb{R}; \mathbb{C})$ is generated by $f\circ \gamma$.
With the help of $U_\gamma$ we introduce the Sobolev space of order $s \in \mathbb{R}$ on $\Gamma$  by
\begin{equation*}
  H^s(\Gamma) := \big\{f \in \mathcal{S}'(\Gamma; \mathbb{C}):  U_\gamma f \in H^s(\mathbb{R}; \mathbb{C}) \big\}.
\end{equation*}
In the main part of the paper we will often use the mapping $\Lambda$ defined by
\begin{equation} \label{def_Lambda}
  (\mathcal{F} U_\gamma \Lambda f)(p) = \sqrt[4]{1+p^2} \, (\mathcal{F} U_\gamma f)(p), \quad  p \in \mathbb{R},
\end{equation}
where $\mathcal{F}$ is the Fourier transform in $\mathcal{S}'(\mathbb{R}; \mathbb{C})$.
Evidently, for any $s \in \mathbb{R}$ the operator $\Lambda: H^s(\Gamma; \mathbb{C}) \rightarrow H^{s-1/2}(\Gamma; \mathbb{C})$ is well defined and bijective and for $s=\frac{1}{2}$ 
it can be viewed as an unbounded self-adjoint operator in $L^2(\Gamma; \mathbb{C})$ with dense domain $H^{1/2}(\Gamma; \mathbb{C})$.

If $\mathcal{H}$ is a Hilbert space and $A$ is a closed operator in $\mathcal{H}$, then we denote its domain of definition, its range, and its kernel by $\dom A$, $\ran A$, and $\ker A$, respectively. Furthermore, we use the symbols $\sigma(A)$, $\sigma_\textup{p}(A)$, and $\rho(A)$ for the spectrum, the point spectrum, and the resolvent set of $A$, respectively. If $A$ is self-adjoint, then $\sigma_{\textup{disc}}(A)$ and $\sigma_{\textup{ess}}(A)$ denote the discrete and the essential spectrum of~$A$, respectively.

\section{Main results} \label{section_main_results}

\subsection{Definition of $A_{\eta, \tau, \lambda}^{m, c}$ and confinement}

Throughout the paper, we assume that the support of the singular interaction is a smooth curve that does not intersect with itself and that is straight outside a compact set. To make this mathematically more precise, we make the following assumption:

\begin{hypothesis} \label{hypothesis_Gamma}
Assume that $\gamma: \mathbb{R} \rightarrow \mathbb{R}^2$ is a $C^\infty$-smooth and injective map that satisfies $|\dot{\gamma}(s)| = 1$ for all $s \in \mathbb{R}$ and 
\begin{equation} \label{arc_length_Gamma_intro}
  \gamma(s) = \begin{cases} (a_- s + b_-, c_- s + d_-) &\text{ for } s < -M, \\ (a_+ s + b_+, c_+ s + d_+) &\text{ for } s > M, \end{cases}
\end{equation}
for suitable constants $M>0$ and $a_\pm, b_\pm, c_\pm, d_\pm \in \mathbb{R}$ with $a_\pm^2 + c_\pm^2 = 1$ and $(a_-, c_-) \neq -(a_+, c_+)$. We set $\Gamma := \ran \gamma$ and denote by $\nu(x)$ the unit normal vector at $x = \gamma(s) \in \Gamma$, $s \in \mathbb{R}$, that is given by $\nu(\gamma(s)) = (\dot{\gamma}_2(s), -\dot{\gamma}_1(s))$. The set $\Omega_+$ is the domain with boundary $\Gamma$ such that $\nu$ is pointing outwards of $\Omega_+$ and we define $\Omega_- := \mathbb{R}^2 \setminus \overline{\Omega_+}$.
\end{hypothesis}

If not mentioned differently, we always assume that $\Gamma$ satisfies the above hypothesis. The condition $(a_-, c_-) \neq -(a_+, c_+)$ in Hypothesis~\ref{hypothesis_Gamma} means that the two asymptotes of $\Gamma$ are not parallel rays (of the same orientation) as $t\to\pm\infty$. Note that for $\gamma$ in Hypothesis~\ref{hypothesis_Gamma} there exists $C_1 > 0$ such that for all $s, t \in \mathbb{R}$ 
\begin{equation} \label{equation_bi_Lipschitz}  
  C_1 |s - t| \leq |\gamma(s) - \gamma(t)| \leq |s - t|
\end{equation}
holds. The second inequality in \eqref{equation_bi_Lipschitz} is clear. To see the first one consider the function $(s,t)\mapsto r(s,t):=\frac{|\gamma(s)-\gamma(t)|}{|s-t|}$, which is positive and continuous, being extended to $r(s,s)=1$ as $\Gamma$ is smooth by assumption. In the first and the third quadrant of the $(s,t)$ plane we have $r(s,t)=1$ for all $s,t$ large enough. Moreover, for any $a\in[0,+\infty)$,
$$\lim_{s\to\pm\infty}\inf_{\pm t\in[0,a]}r(s,t)=\lim_{s\to\pm\infty}\frac{|\gamma(s)|}{|s|}=1,$$
and similarly with the roles of $s$ and $t$ interchanged. 
In the other two quadrants we have $r(s,t)\leq 1$ and it is not difficult to check that
\begin{equation*}   
\lim_{|s|\to\infty} r(s,-\mu s) = \frac{\sqrt{1+\mu^2-2\mu\cos\theta}}{1+\mu}
\end{equation*}
holds for $\mu>0$, where $\theta$ is the angle between the asymptotes; the right-hand side reaches its minimum $\sin \frac{\theta}{2}$ at the axis of the quadrants, $\mu=1$. We conclude that there is a compact subset $M$ of the plane outside of which we have $r(s,t)>\frac12 \sin \frac{\theta}{2}$. In view of the compactness, the minimum is attained within $M$ and $C_2:= \min_M r(s,t)>0$, because the curve does not intersect itself; it is then enough to choose $C_1=\min\{C_2,\frac12 \sin \frac{\theta}{2}\}$.  

In the following we use the notation $f_\pm := f\upharpoonright \Omega_\pm$ for a function $f$ that is defined on $\mathbb{R}^2$. Recall that $\sigma_1, \sigma_2, \sigma_3$ are the $\mathbb{C}^{2 \times 2}$-valued Pauli spin matrices defined in~\eqref{def_Pauli_matrices} and $\sigma_0$ is the $2\times 2$ identity matrix. 
 We define the spaces
\begin{equation*}
  H(\sigma, \Omega_\pm) := \big\{ f_\pm \in L^2(\Omega_\pm; \mathbb{C}^2): (\sigma\cdot \nabla) f_\pm \in L^2(\Omega_\pm; \mathbb{C}^2) \big\}
\end{equation*}
and note that for $f_\pm \in H(\sigma; \Omega_\pm)$ the trace $T^D_\pm f_\pm$ exists in the Sobolev space $H^{-1/2}(\Gamma; \mathbb{C}^2)$; cf. Section~\ref{section_A_0} for details. This allows us to define for $m, \eta, \tau, \lambda \in \mathbb{R}$ and $c > 0$ in $L^2(\mathbb{R}^2; \mathbb{C}^2)$ the operator
\begin{equation} \label{def_A_eta_tau_intro}
  \begin{split}
    &\, \, \, \, \, \, \, \,  \, A_{\eta, \tau, \lambda}^{m, c} f := \big(-i c (\sigma \cdot \nabla) + m c^2 \sigma_3 \big) f_+ \oplus \big(-i c (\sigma \cdot \nabla) + m c^2 \sigma_3 \big) f_-, \\
    &\dom A_{\eta, \tau, \lambda}^{m, c} := \bigg\{ f = f_+ \oplus f_- \in H(\sigma, \Omega_+) \oplus H(\sigma, \Omega_-): \\
    &\quad  -i c (\sigma \cdot \nu) (T^D_+f_+ - T^D_- f_-) = \frac{1}{2} (\eta \sigma_0 + \tau \sigma_3 + \lambda i (\sigma \cdot \nu) \sigma_3) (T^D_+ f_+ + T^D_- f_-) \bigg\}.
  \end{split}
\end{equation}
Using integration by parts, it can be seen that the operator in \eqref{def_A_eta_tau_intro} is the rigorous mathematical definition of the formal expression in \eqref{eq:formal_expr}, see, e.g., \cite{BoPi_21}. 

A remarkable property of $A_{\eta, \tau, \lambda}^{m, c}$ is that it decomposes for special combinations of $\eta, \tau, \lambda$ to the orthogonal sum of two Dirac operators in $L^2(\Omega_\pm; \mathbb{C}^2)$ with boundary conditions. This phenomenon, which was observed originally in the three-dimensional setting and later in the current two-dimensional case (but always with  bounded interaction supports) in \cite{AMV15, BEHL19_1, BHOP20, BHSS22, CLMT21, DES89}, is known as confinement and means that the $\delta$-potential is, in this case, impenetrable for the quantum particle. The mentioned boundary conditions are known as quantum dot boundary conditions and play an important role in the mathematical analysis of graphene \cite{BFSV17}. Note that in all studied cases the confinement occurs if and only if the parameter
\begin{equation} \label{eq:d_def}
d:=\eta^2 - \tau^2 - \lambda^2
\end{equation}
satisfies $d  = -4 c^2$ and the same is also true in the current setting.

\begin{prop} \label{prop_confinement}
  Let $\Gamma$ be as in Hypothesis~\ref{hypothesis_Gamma}, $m, \eta, \tau, \lambda \in \mathbb{R}$ and $c >0$. Then, the following is true:
  \begin{itemize}
    \item[(i)] If $d \neq -4 c^2$, then there exists a map $Q_{\eta, \tau, \lambda}: \Gamma \rightarrow \mathbb{C}^{2 \times 2}$ depending on the parameters $\eta,\tau,\lambda$ such that $Q_{\eta, \tau, \lambda}$ is pointwise invertible and $f \in \dom A_{\eta,\tau,\lambda}^{m, c}$ if and only if $f \in H(\sigma; \Omega_+) \oplus H(\sigma; \Omega_-)$ and
    \begin{equation*} 
      T^D_+ f_+ = Q_{\eta, \tau, \lambda} T^D_- f_-.
    \end{equation*}
    \item[(ii)] If $d = -4 c^2$, then $A_{\eta, \tau, \lambda}^{m, c} = A_{\eta,\tau,\lambda}^+ \oplus A_{\eta,\tau,\lambda}^-$, where $A_{\eta,\tau,\lambda}^\pm$ are the operators acting in $L^2(\Omega_\pm; \mathbb{C}^2)$ given by
    \begin{equation*} 
      \begin{split}
        A_{\eta,\tau,\lambda}^\pm f &= \big(-i c (\sigma \cdot \nabla) + m c^2 \sigma_3 \big) f, \\
        \dom A_{\eta,\tau,\lambda}^\pm &= \big\{ f \in H(\sigma; \Omega_\pm): \big( 2 c \sigma_0 \mp i (\sigma \cdot \nu) (\eta \sigma_0 + \tau \sigma_3 +  \lambda i(\sigma \cdot \nu) \sigma_3) \big) T^D_\pm f = 0 \big\}.
      \end{split}
    \end{equation*}
  \end{itemize}
\end{prop}

The proof of this result can be done in exactly the same way as, e.g., in \cite[Theorem~5.5]{AMV15}, \cite[Lemma~4.1]{BHOP20}, \cite[Lemma~5.11]{BHSS22}, or \cite[Theorem~2.3]{CLMT21}, and hence, we omit it in this paper.

\subsection{Self-adjointness and essential spectrum of $A_{\eta, \tau, \lambda}^{m, c}$}

In the following theorem the self-adjointness and the essential spectrum of $A_{\eta, \tau, \lambda}^{m, c}$ are discussed. Note that, in a similar way as it was observed in \cite{BH19, BHOP20, BHSS22, CLMT21, OV18} for compact interaction supports, there is a loss of Sobolev regularity in the operator domain in the case of critical interaction strengths $(\tfrac{d}{4} - c^2)^2 - \lambda^2 c^2 = 0$. Moreover, in the particular case of $\Gamma$ being a straight line the spectrum of $A_{\eta, \tau, \lambda}^{m, c}$ was computed in \cite{BHT23}; Theorem~\ref{theorem_intro} says that the essential spectrum is preserved under the geometric perturbations as long as $\Gamma$ is straight outside a compact set. We point out that to the best of our knowledge item~(i) of the following theorem on the Sobolev regularity in $\dom A_{\eta, \tau, \lambda}^{m, c}$ was not known in this generality before even for the straight line.

\begin{thm} \label{theorem_intro}
  Let $m, \eta, \tau, \lambda \in \mathbb{R}$, $c > 0$, $d$ be given in \eqref{eq:d_def}, and $\Gamma$ be as in Hypothesis~\ref{hypothesis_Gamma}. Then, the operator $A_{\eta, \tau, \lambda}^{m, c}$ is self-adjoint in $L^2(\mathbb{R}^2; \mathbb{C}^2)$ and the following holds:
  \begin{itemize}
    \item[(i)] For the domain of $A_{\eta, \tau, \lambda}^{m, c}$ one has:
    \begin{itemize}
      \item[(a)] If $(\tfrac{d}{4} - c^2)^2 - \lambda^2 c^2 \neq 0$, then $\dom A_{\eta, \tau, \lambda}^{m, c} \subset H^1(\mathbb{R}^2 \setminus \Gamma; \mathbb{C}^2)$.
      \item[(b)] If $(\tfrac{d}{4} - c^2)^2 - \lambda^2 c^2 = 0$, then $\dom A_{\eta, \tau, \lambda}^{m, c} \not\subset H^s(\mathbb{R}^2 \setminus \Gamma; \mathbb{C}^2)$ for all $s > 0$.
    \end{itemize}
    \item[(ii)] For the essential spectrum of $A_{\eta, \tau, \lambda}^{m, c}$ one has:
    \begin{itemize}
      \item[(a)] If $d=4 c^2$ and $\lambda \neq 0$, then 
      \begin{equation*}
        \sigma_{\textup{ess}}(A_{\eta, \tau, \lambda}^{m, c}) = \sigma(A_{\eta, \tau, \lambda}^{m, c}) = \mathbb{R}.
      \end{equation*}
      \item[(b)] If $d=4 c^2$ and $\lambda = 0$, then 
      \begin{equation*}
        \sigma_{\textup{ess}}(A_{\eta, \tau, \lambda}^{m, c}) = (-\infty, -|m| c^2] \cup \left\{ -\frac{\tau}{\eta} m c^2 \right \} \cup [|m| c^2, +\infty).
      \end{equation*}
      \item[(c)] If $d \neq 4 c^2$, then $$(-\infty, -|m| c^2] \cup [|m| c^2, +\infty) \subset \sigma_{\textup{ess}}(A_{\eta, \tau, \lambda}^{m, c})$$ and 
      $\sigma_\textup{ess}(A_{\eta, \tau, \lambda}^{m, c}) \cap (-|m| c^2, |m| c^2)$ coincides with the set 
      \begin{equation*}
         \overline{\left\{ z_\pm(k): (d - 4c^2) \left(\frac{\eta}{c^2} z_\pm(k) + \lambda k + \tau m\right) > 0 \right\}} \cap (-|m| c^2, |m| c^2),
      \end{equation*}
      where $z_\pm(k)$, $k \in \mathbb{R}$, is given by 
      \begin{equation*} 
        \begin{split}
          z_\pm(k) &= \frac{1}{\tfrac{\eta^2}{c^2} + \big( \tfrac{d}{4 c^2} - 1 \big)^2} \bigg( -\eta (\lambda k + \tau m) \\
          &\,\, \pm \left| \frac{d}{4 c^2}- 1 \right| \sqrt{\big( \tau^2 c^2 + \big( \tfrac{d}{4} + c^2 \big)^2 \big) k^2 - 2 \lambda \tau m k c^2 + \big( \lambda^2 c^2  + \big(  \tfrac{d}{4} + c^2 \big)^2 \big) m^2} \bigg).
        \end{split}
      \end{equation*}
    \end{itemize}
  \end{itemize}
  Moreover, if $\Gamma$ is the straight line, then $\sigma(A_{\eta, \tau, \lambda}^{m, c})=\sigma_\textup{ess}(A_{\eta, \tau, \lambda}^{m, c})$, i.e. the spectrum of $A_{\eta, \tau, \lambda}^{m, c}$ is fully characterized by \textup{(ii)} in this case.
\end{thm}

We point out that an interesting spectral transition appears in the case when $d=4 c^2$: If $\lambda \neq 0$, then always $\sigma(A_{\eta, \tau, \lambda}^{m, c}) = \mathbb{R}$, but for $\lambda = 0$ the essential spectrum within the interval $(-|m| c^2, |m| c^2)$ breaks down to the single point $-\frac{\tau}{\eta} m c^2$.

Theorem~\ref{theorem_intro} will be proved in Section~\ref{section_proof} with the help of a suitable boundary triple. In Subsection~\ref{section_krein} we will also show a variant of the Birman-Schwinger principle, a Krein type resolvent formula, and an isospectral relation that are useful in the analysis of $A_{\eta, \tau, \lambda}^{m, c}$.

\subsection{Special cases}

The next results follow immediately from Theorem~\ref{theorem_intro} (and thus, no proofs of these corollaries are given), but we prefer to state them here as they are more explicit than Theorem~\ref{theorem_intro} and of particular interest in applications. 
First, we discuss the case of purely electrostatic interactions, i.e. when $\eta \in \mathbb{R}$ and $\tau = \lambda = 0$. Note that, as in the case of the straight line \cite{BHT22}, there appears a spectral transition at $\eta  =\pm 2 c$, where an interval of continuous spectrum of $A_{\eta, 0, 0}^{m, c}$ collapses to a single point in the essential spectrum. 

\begin{cor} \label{corollary_electrostatic}
  Let $\eta \in \mathbb{R}$. Then, one has 
  \begin{equation*} 
    \sigma_{\textup{ess}}(A_{\eta,0,0}^{m, c}) = \begin{cases} 
        (-\infty, \frac{4 c^2 - \eta^2}{4 c^2 + \eta^2}|m| c^2]\cup[|m| c^2,+\infty) & \text{ for }\eta\in(-\infty,-2c),\\
        (-\infty,-|m| c^2] \cup \{ 0 \} \cup [|m| c^2,+\infty) & \text{ for }\eta=-2 c,\\
        (-\infty, -|m| c^2] \cup [\frac{4 c^2 - \eta^2}{4 c^2 + \eta^2}|m| c^2,+\infty) & \text{ for }\eta\in(-2c,0),\\
        (-\infty,-|m| c^2] \cup [|m| c^2,+\infty) & \text{ for }\eta=0,\\
        (-\infty, \frac{\eta^2 - 4 c^2}{\eta^2 + 4 c^2} |m| c^2]\cup[|m| c^2,+\infty) & \text{ for }\eta\in(0,2c),\\
        (-\infty,-|m| c^2] \cup \{ 0 \} \cup [|m| c^2,+\infty) & \text{ for }\eta=2 c,\\
        (-\infty, -|m| c^2] \cup [\frac{\eta^2 - 4 c^2}{\eta^2 + 4 c^2}|m| c^2, +\infty) & \text{ for }\eta\in(2c,+\infty). \end{cases}
  \end{equation*}
\end{cor}

For purely Lorentz scalar interactions, i.e. for $\tau \in \mathbb{R}$ and $\eta = \lambda = 0$, the following result holds. Note that for $\tau m \geq 0$ the essential spectrum of $A_{0, \tau, 0}^{m, c}$ coincides with the spectrum of the free Dirac operator.

\begin{cor} \label{corollary_Lorentz_scalar}
  Let $\tau \in \mathbb{R}$. Then, one has 
  \begin{equation*} 
    \sigma_{\textup{ess}}(A_{0, \tau,0}^{m, c}) = \begin{cases} 
        (-\infty,-|m| c^2] \cup [|m| c^2, +\infty) & \text{ for } \tau m\geq 0,\\
        (-\infty, -\frac{|4 c^2 - \tau^2|}{4 c^2 + \tau^2} |m| c^2]\cup[\frac{|4 c^2 - \tau^2|}{4 c^2 + \tau^2} |m| c^2,+\infty) & \text{ for } \tau m<0. \end{cases}
  \end{equation*}
\end{cor}

Finally, we consider the case of purely anomalous magnetic interactions, i.e. when $\lambda \in \mathbb{R}$ and $\eta = \tau = 0$. Note that in this configuration the essential spectrum of $A_{0, 0, \lambda}^{m, c}$ always coincides with the spectrum of the free Dirac operator.

\begin{cor} \label{corollary_magnetic}
  Let $\lambda \in \mathbb{R}$. Then, one has 
  \begin{equation*} 
    \sigma_{\textup{ess}}(A_{0, 0, \lambda}^{m, c}) = (-\infty,-|m| c^2] \cup [|m| c^2, +\infty).
  \end{equation*}
\end{cor}



\subsection{Geometrically induced bound states}

In the present paper we prove the existence of geometrically induced bound states in two situations: Firstly, for $A_{\eta/2, \eta/2, 0}^{m, c}$ for large values of $c$ with the help of the nonrelativistic limit, and secondly for $A_{0, \tau, 0}^{m, c}$ by using the test functions from \cite{FHL23}.

The nonrelativistic limit is one way to relate a relativistic Dirac operator $A^{m, c}$ with the Schr\"odinger operator $H$ that describes the same physical system \cite[Chapter~6]{T92}. To compute the nonrelativistic limit, one has to subtract the particle's rest energy $m c^2$, which is a purely relativistic quantity, from the total energy and compute the limit in the operator norm of the resolvent of the resulting operator $A^{m, c} - m c^2$ for $c \rightarrow +\infty$. The expected limit is then the resolvent of $H$ times a projection onto the upper component. This result is of interest by its own, as it gives a physical interpretation of the model, but taking properties of convergence of the resolvents into account, this allows to relate the spectral properties of $H$ to those of $A^{m, c}$ for large values of $c$.

In this paper we consider a simple case and compute the nonrelativistic limit of $A_{\eta/2, \eta/2, 0}^{m, c}$, see also Remark~\ref{remark_nonrelativistic_limit} for a more general consideration. In the following, we assume that $\Gamma$ is given as in Hypothesis~\ref{hypothesis_Gamma} and that $m > 0$. First, we have to introduce the expected limit operator, a Schr\"odinger operator with a $\delta$-potential supported on $\Gamma$.
Define for $\eta \in \mathbb{R}$ in $L^2(\mathbb{R}^2; \mathbb{C})$ the sesquilinear form
\begin{equation} \label{def_Schroedinger_delta}
  \mathfrak{h}_\eta[f, g] := \frac{1}{2 m} (\nabla f, \nabla g)_{L^2(\mathbb{R}^2; \mathbb{C})} + \eta (T^D f, T^D g)_{L^2(\Gamma; \mathbb{C})}, \quad \dom \mathfrak{h}_\eta = H^1(\mathbb{R}^2; \mathbb{C}),
\end{equation}
where $T^D: H^1(\mathbb{R}; \mathbb{C}) \rightarrow H^{1/2}(\Gamma; \mathbb{C})$ is the Dirichlet trace operator.
It is well-known that $\mathfrak{h}_\eta$ is densely defined, symmetric, and  closed \cite{BEKS94, EI01}. Hence, by the first representation theorem \cite[Chapter~VI, Theorem~2.1]{K95} there exists a unique self-adjoint operator $H_\eta$ that is associated with $\mathfrak{h}_\eta$. The operator $H_\eta$  can be regarded as the rigorous way to interpret the formal expression  $-\frac{1}{2 m} \Delta + \eta \delta_\Gamma$. With this notation we can formulate the result about the nonrelativistic limit of $A_{\eta/2, \eta/2, 0}^{m, c}$, which yields, combined with a result from \cite{EI01}, a statement about the existence of geometrically induced bound states:

\begin{thm} \label{theorem_nonrelativistic_limit}
  Let $\eta\in \mathbb{R}$, $m > 0$, and $\Gamma$ be as in Hypothesis~\ref{hypothesis_Gamma}. Then, there exists $z_0 < 0$ such that  every $z\in(-\infty,z_0)$ belongs to $\rho(A_{\eta/2, \eta/2, 0}^{m, c}-mc^2) \cap \rho(H_\eta)$, whenever $c$ is above a $z$-dependent threshold value $c_z$. Moreover, there exists a constant $K > 0$ such that
  \begin{equation*}
    \left\| \big(A_{\eta/2, \eta/2, 0}^{m, c} - (z + m c^2)\big)^{-1} - (H_\eta - z)^{-1} \begin{pmatrix} 1 & 0 \\ 0 & 0 \end{pmatrix} \right\|_{L^2(\mathbb{R}^2; \mathbb{C}^2) \rightarrow L^2(\mathbb{R}^2; \mathbb{C}^2)} \leq \frac{K}{c}
  \end{equation*}
  for all $c>c_z$.
  In particular, if $\Gamma$ is not the straight line, $\eta < 0$, and $c$ is sufficiently large,  then $\sigma_{\textup{disc}}(A_{\eta/2, \eta/2, 0}^{m, c}) \neq \emptyset$.
\end{thm}

Theorem~\ref{theorem_nonrelativistic_limit} will  be proved in Section~\ref{section_nonrelativistic_limit} with the help of the Krein type resolvent formula for $A_{\eta/2, \eta/2, 0}^{m, c}$ provided in Section~\ref{section_proof}. In this resolvent formula, convergence of each term appearing there can be proved, when $c \rightarrow +\infty$, which yields 
the claim about the nonrelativistic limit. The statement about the discrete spectrum is then a simple consequence of the nonrelativistic limit, the non-emptiness of $\sigma_\textup{disc}(H_\eta)$ under the stated assumptions shown in \cite{EI01}, and general results that are known for operators that converge in the norm resolvent sense.

Another way to show the existence of geometrically induced bound states is to transfer a recent result from \cite{FHL23}.
Define for $\omega \in (0, \frac{\pi}{2})$ the broken line $\widetilde{\Gamma}_\omega$ of opening angle $2 \omega$ by
\begin{equation} \label{def_broken_line}
  \widetilde{\Gamma}_\omega := \big\{ (r \cos(\omega), r \sin(\omega)): r > 0 \big\} \cup \big\{ (r \cos(\omega), -r \sin(\omega)): r > 0 \big\}.
\end{equation}
It was shown in \cite{FHL23} that a Dirac operator with a Lorentz scalar $\delta$-shell interaction supported on $\widetilde{\Gamma}_\omega$ of strength $\tau \in (-\infty, 0) \setminus \{ -2 \}$ always has non-empty discrete spectrum, if $\omega$ is sufficiently small. This can be transferred to obtain an associated result for $A_{0, \tau, 0}^{m, c}$, if the interaction support $\Gamma$ is a compact smooth perturbation of $\widetilde{\Gamma}_\omega$. We remark that any curve $\Gamma$ satisfying Hypothesis~\ref{hypothesis_Gamma}, that is not a compact perturbation of the straight line, is, up to a translation and rotation, of this form and that the following result, in contrast to the one in Theorem~\ref{theorem_nonrelativistic_limit}, can be shown for arbitrary $c > 0$.

\begin{thm} \label{theorem_frymark_holzmann_lotoreichik} 
  Let $\tau \in (-\infty, 0) \setminus \{-2 \}$ and $\Gamma$ be as in Hypothesis~\ref{hypothesis_Gamma} such that $\Gamma$  coincides with $\widetilde{\Gamma}_\omega$ away from a compact set in $\mathbb{R}^2$, and  $m, c>0$. Then, for any $N \in \mathbb{N}$ there exists an angle $\omega_\star \in (0, \frac{\pi}{2})$ depending on $N, \tau$ such that for all $\omega \in (0, \omega_\star]$ the operator $A_{0, \tau, 0}^{m, c}$ has at least $N$ discrete eigenvalues with multiplicities taken into account.
\end{thm}

Theorem~\ref{theorem_frymark_holzmann_lotoreichik} is also shown in Section~\ref{section_nonrelativistic_limit} with the help of the min-max principle applied to the non-negative and self-adjoint operator $(A_{0, \tau, 0}^{m, c})^2$ and a family of test functions  constructed in \cite{FHL23} that  are supported in a region where $\Gamma$ and $\widetilde{\Gamma}_\omega$ coincide, which allows to transfer the result from \cite{FHL23} to the situation studied in this paper.

\section{Preliminaries} \label{section_preliminaries}

In this section we collect several results, that are needed to prove the main claims made on $A_{\eta, \tau, \lambda}^{m, c}$ in Section~\ref{section_main_results}. First, in Section~\ref{section_A_0} we summarize some well-known properties of the free Dirac operator $A_0^{m, c}$ and in Section~\ref{section_integral_operators} we state some properties of integral operators that are associated with the integral kernel of the resolvent of $A_0^{m, c}$. In Section~\ref{section_bt} we construct a boundary triple that is useful to study $A_{\eta, \tau, \lambda}^{m, c}$. Throughout this section, we always assume that $\Gamma \subset \mathbb{R}^2$ is as in Hypothesis~\ref{hypothesis_Gamma}.

\subsection{The free Dirac operator and function spaces for Dirac operators} \label{section_A_0}

In this section we recall the definition of the free Dirac operator in $\mathbb{R}^2$ and state several related results that are needed in the main part of the paper. Let $\sigma_1, \sigma_2, \sigma_3$ be the Pauli spin matrices defined in~\eqref{def_Pauli_matrices}.
For $m \in \mathbb{R}$ and $c > 0$ the free Dirac operator in $\mathbb{R}^2$ is 
\begin{equation} \label{def_free_op}
  A_0^{m, c} f := -i c (\sigma \cdot \nabla) f + m c^2 \sigma_3 f, \quad \dom A_0^{m, c} = H^1(\mathbb{R}^2; \mathbb{C}^2). 
\end{equation}
We will sometimes make  use of the relation $A_0^{m, c} = c A_0^{mc, 1}$.
With the help of the Fourier transform it is not difficult to show that $A_0^{m, c}$ is self-adjoint in $L^2(\mathbb{R}^2; \mathbb{C}^2)$ and that
\begin{equation*}
  \sigma(A_0^{m, c}) = \sigma_\text{ess}(A_0^{m, c}) = (-\infty, -|m| c^2] \cup [|m| c^2, +\infty).
\end{equation*}
Moreover, for $z \in \rho(A_0^{m, c})$ the resolvent of $A_0^{m, c}$ is given by
\begin{equation*}
  (A_0^{m, c} - z)^{-1} f(x) = \int_{\mathbb{R}^2} G_z^{m, c}(x - y) f(y) d y, \quad f \in L^2(\mathbb{R}^2; \mathbb{C}^2), ~x \in \mathbb{R}^2, 
\end{equation*}
where
\begin{equation} \label{def_G_lambda}
  \begin{split}
    G_z^{m, c}(x) &= \frac{1}{2 \pi c} \sqrt{\tfrac{z^2}{c^2} - (m c)^2} K_1 \left( - i \sqrt{\tfrac{z^2}{c^2} - (m c)^2} | x |\right)  \frac{(\sigma \cdot x)}{ | x | } \\
    &\qquad \qquad + \frac{1}{2 \pi c} K_0 \left(- i \sqrt{\tfrac{z^2}{c^2} - (m c)^2} | x |\right) \left( \frac{z}{c} \sigma_0 + m c \sigma_3\right)
  \end{split}
\end{equation}
and $K_j$ is the modified Bessel function of the second kind and order $j$; cf. \cite{AS84}. We remark that we always choose the complex square root such that for all $w\in\mathbb{C}\setminus[0,+\infty)$ one has $\textup{Im}(\sqrt{w})>0$ and hence $\overline{\sqrt{w}}=-\sqrt{\overline w}$ for $w\in\mathbb{C}\setminus[0,+\infty)$. Note also that
$A_0^{m, c} = c A_0^{mc, 1}$ yields 
\begin{equation} \label{relation_G_z_c}
  (A_0^{m, c} - z)^{-1} = \frac{1}{c} \left( A_0^{mc, 1} - \frac{z}{c} \right)^{-1} \quad \text{and} \quad G_z^{m, c} = \frac{1}{c} G_{z/c}^{m c, 1},\quad z \in \rho(A_0^{m, c}).
\end{equation}

Recall that $\Omega_\pm \subset \mathbb{R}^2$ are the two open sets with joint boundary $\partial \Omega_\pm = \Gamma$, see Hypothesis~\ref{hypothesis_Gamma}, and that we denote elements in $L^2(\Omega_\pm; \mathbb{C}^2)$ with a subscript $\pm$. We consider the spaces
\begin{equation} \label{def_H_sigma}
  H(\sigma, \Omega_\pm) := \big\{ f_\pm \in L^2(\Omega_\pm; \mathbb{C}^2): (\sigma \cdot \nabla) f_\pm \in L^2(\Omega_\pm; \mathbb{C}^2) \big\},
\end{equation}
which are, endowed with the norms 
\begin{equation*}
  \| f_\pm \|_{H(\sigma, \Omega_\pm)}^2 := \| f_\pm \|_{L^2(\Omega_\pm; \mathbb{C}^2)}^2 + \| (\sigma \cdot \nabla) f_\pm \|_{L^2(\Omega_\pm; \mathbb{C}^2)}^2,
\end{equation*}
Hilbert spaces; cf. \cite[Lemma~2.1]{BFSV17} for similar arguments in the case when $\Gamma$ is bounded. An important property of $H(\sigma, \Omega_\pm)$ is the fact that the trace operator can be extended to these spaces. The proof of the next lemma for our geometric setting follows very similar ideas as the ones in \cite[Lemma~2.3]{BFSV17} or \cite[Lemma~2.3]{PvdB22}.

\begin{lem} \label{lemma_trace_theorem}
  The map $T^D_{\pm, 0}: H^1(\Omega_\pm; \mathbb{C}^2) \rightarrow H^{1/2}(\Gamma; \mathbb{C}^2), T^D_{\pm, 0} f = f |_\Gamma$, extends uniquely by continuity to a bounded linear operator $T^D_\pm: H(\sigma, \Omega_\pm) \rightarrow H^{-1/2}(\Gamma; \mathbb{C}^2)$.
\end{lem}

In the main part of the paper we will work with the densely defined closed symmetric restriction $S$ of $A_0^{m, c}$ onto $H^1_0(\mathbb{R}^2 \setminus \Gamma; \mathbb{C}^2)$, that is, 
\begin{equation} \label{def_S}
  S f = -i c (\sigma \cdot \nabla) f + m c^2 \sigma_3 f, \quad \dom S = H^1_0(\mathbb{R}^2 \setminus \Gamma; \mathbb{C}^2),
\end{equation}
and its adjoint
\begin{equation} \label{def_S_star}
  \begin{split}
    S^* f &= \big(-i c (\sigma \cdot \nabla) + m c^2 \sigma_3 \big) f_+ \oplus (-i c (\sigma \cdot \nabla) + m c^2 \sigma_3 \big) f_-, \\
    \dom S^* &= H(\sigma, \Omega_+) \oplus H(\sigma, \Omega_-).
  \end{split}
\end{equation}

\subsection{Integral operators for the Dirac operator} \label{section_integral_operators}

In this section, we introduce two families of integral operators  that will play an important role in the analysis of $A_{\eta, \tau, \lambda}^{m, c}$. 
For this, recall the fundamental solution $G_z^{m, c}$ in~\eqref{def_G_lambda} and note that  $\Gamma$ given in Hypothesis~\ref{hypothesis_Gamma} belongs to the class of curves considered in \cite{BHS23}. 

Following \cite[Appendix~C]{BHS23}, we introduce for 
$$z \in \rho(A_0^{m, c}) = \mathbb{C} \setminus \bigl((-\infty, -|m| c^2] \cup [|m| c^2, +\infty)  \bigr)$$ 
the operator
\begin{equation} \label{Phi_star}
  (\Phi_z^{m, c})^* := T^D (A_0^{m, c}-\overline{z})^{-1}: L^2(\mathbb{R}^2; \mathbb{C}^2) \rightarrow H^{1/2}(\Gamma; \mathbb{C}^2),
\end{equation}
which is well-defined and bounded due to the mapping properties of the Dirichlet trace $T^D: H^1(\mathbb{R}^2) \rightarrow H^{1/2}(\Gamma)$ and the resolvent of $A_0^{m, c}$.
Hence, also its anti-dual
\begin{equation} \label{def_Phi_z}
  \Phi_z^{m, c} := \big( (\Phi_z^{m, c})^* \big)': H^{-1/2}(\Gamma; \mathbb{C}^2) \rightarrow L^2(\mathbb{R}^2; \mathbb{C}^2)
\end{equation}
is well-defined and bounded. 
Using \eqref{Phi_star} and $G_z^{m, c}(x)^* = G_{\overline{z}}^{m, c}(-x)$ one finds that the action of $\Phi_z^{m, c}$ on $\varphi \in L^2(\Gamma; \mathbb{C}^2)$ is
\begin{equation*}
  \Phi_z^{m, c} \varphi(x) = \int_\Gamma G_z^{m, c}(x-y) \varphi(y) d \sigma(y), \quad x \in \mathbb{R}^2 \setminus \Gamma.
\end{equation*}
Moreover, as $\ker(\Phi_z^{m, c})^* = \ran (S - \overline{z})$, where $S$ is the symmetric operator in~\eqref{def_S}, and $\ran (\Phi_z^{m, c})^*$ is closed, we get that
\begin{equation} \label{equation_Phi_z_solution}
  \ran \Phi_z^{m, c} = \ker (S^* - z) \subset H(\sigma, \Omega_+) \oplus H(\sigma, \Omega_-).
\end{equation}
Since by~\eqref{relation_G_z_c} one has $\Phi_z^{m, c} = \frac{1}{c} \Phi_{z/c}^{m c, 1}$ and $\Phi_{z/c}^{m c, 1} = \Phi_{z/c}$ in the notation of \cite{BHS23} (when $m$ is replaced by $mc$), we conclude from \cite[Proposition~2.8]{BHS23} that $\Phi_z^{m, c}$ has the bounded restriction
\begin{equation} \label{mapping_properties_Phi_z}
  \widetilde{\Phi}_z^{m, c}: H^{1/2}(\Gamma; \mathbb{C}^2) \rightarrow H^1(\mathbb{R}^2 \setminus \Gamma; \mathbb{C}^2).
\end{equation}

The boundary integral operator formally given by
\begin{equation} \label{def_C_z_formal}
  \mathcal{C}_z^{m, c} \varphi(x) := \text{p.v.} \int_\Gamma G_z^{m, c}(x-y) \varphi(y) d \sigma(y), \quad x \in \Gamma, ~\varphi \in L^2(\Gamma; \mathbb{C}^2),
\end{equation}
where $z \in \rho(A_0^{m, c})$ and the integral is understood as the principal value, will be used throughout this paper. Since $\mathcal{C}_z^{m,c} = \frac{1}{c} \mathcal{C}_{z/c}^{m c, 1}$ holds by~\eqref{relation_G_z_c} and $\mathcal{C}_{z/c}^{m c, 1} = \mathcal{C}_{z/c}$ in the notation of \cite{BHS23} (when $m$ is replaced by $mc$), we get with \cite[Proposition~2.9]{BHS23} that $\mathcal{C}_z^{m, c}$ gives rise to a bounded operator in $H^r(\Gamma; \mathbb{C}^2)$ for $r \in [-\frac{1}{2}, \frac{1}{2}]$. Moreover, as $G_z^{m, c}(x)^* = G_{\overline{z}}^{m, c}(-x)$, $x \in \mathbb{R}^2 \setminus \{ 0 \}$, it can be shown that the adjoint of $\mathcal{C}_z^{m, c}$, viewed as an operator in $L^2(\Gamma; \mathbb{C}^2)$, is $(\mathcal{C}_z^{m, c})^* = \mathcal{C}_{\overline{z}}^{m, c}$.

Next, we explain that in a suitable sense $\mathcal{C}_z^{m, c}$ can be regarded as a compact perturbation of the same mapping $\mathcal{C}_z^{m, c, \Sigma}$ on the straight line $\Sigma = \mathbb{R} \times \{ 0 \}$, i.e. of the operator that acts on $\varphi \in L^2(\Sigma; \mathbb{C}^2)$ as
\begin{equation} \label{def_C_z_formal_straight_line}
  \mathcal{C}_z^{m, c, \Sigma} \varphi(s, 0) := \text{p.v.} \int_{-\infty}^\infty G_z^{m, c}(s-t, 0) \varphi(t, 0) d t, \quad s \in \mathbb{R}.
\end{equation}
In the formulation of the next proposition, the multiplication operator with the matrix-valued function
\begin{equation} \label{def_V}
  V = \begin{pmatrix} 1 & 0 \\ 0 & \overline{\boldsymbol{t}} \end{pmatrix},
\end{equation}
where $\boldsymbol{t} = t_1 + i t_2$ with $(t_1,t_2) = \dot{\gamma} \circ \gamma^{-1}$ being the unit tangent vector at $\Gamma$, will play an important role. Since $\boldsymbol{t}$ is $C^\infty$-smooth and $|\boldsymbol{t}|=1$, one clearly has that for any $r \in \mathbb{R}$ the multiplication with $V$ gives rise to a bounded and bijective operator in $H^r(\Gamma; \mathbb{C}^2)$. Moreover, recall the definition of the map $U_g$ from~\eqref{def_U} for any arc-length parametrization $g$ of an unbounded curve. The technical proof of Proposition~\ref{proposition_C_z_Gamma} is given in Appendix~\ref{appendix_int_op}.

\begin{prop} \label{proposition_C_z_Gamma}
  Let $z \in \rho(A_0^{m, c})=\mathbb{C} \setminus ((-\infty, -|m| c^2] \cup [|m| c^2, +\infty))$, $\Gamma$ be as in Hypothesis~\ref{hypothesis_Gamma}, and $x(t) = (t,0)$ be a parametrization of $\Sigma$. Then the operator
  \begin{equation*}
    U_{\gamma} V (\sigma \cdot \nu) \mathcal{C}_z^{m, c} (\sigma \cdot \nu) V^* U_\gamma^{-1} - U_x \sigma_2 \mathcal{C}_z^{m, c, \Sigma} \sigma_2 U_x^{-1}: H^{-1/2}(\mathbb{R}; \mathbb{C}^2) \rightarrow H^{1/2}(\mathbb{R}; \mathbb{C}^2)
  \end{equation*}
  is compact. 
\end{prop}

\begin{remark}
  In fact, we will show in the proof of Proposition~\ref{proposition_C_z_Gamma} that the map $U_{\gamma} V (\sigma \cdot \nu) \mathcal{C}_z^{m, c} (\sigma \cdot \nu) V^* U_\gamma^{-1} - U_x \sigma_2 \mathcal{C}_z^{m, c, \Sigma} \sigma_2 U_x^{-1}$ can be extended, for any $r \in [-1, 0]$, to a compact operator from $H^{r}(\mathbb{R}; \mathbb{C}^2)$ to $H^{r+1}(\mathbb{R}; \mathbb{C}^2)$.
\end{remark}

Finally, we state a useful variant of the Plemelj-Sokhotskii formula, which relates the trace of $\Phi_z^{m, c} \varphi$ to $\mathcal{C}_z^{m, c} \varphi$; cf. \cite[Proposition~2.9~(ii)]{BHS23}. For a direct proof of this formula one can also follow line by line the proof of \cite[Proposition~3.4]{BHOP20} using the classical Plemelj-Sokhotskii formula for the Cauchy transform on unbounded curves, see, e.g.,  \cite[Proposition~4.6.1]{G14}. 

\begin{lem} \label{lemma_Plemelj_Sokhotskii}
  Let $z \in \rho(A_0^{m, c})=\mathbb{C} \setminus ((-\infty, -|m| c^2] \cup [|m| c^2, +\infty))$. Then one has for all $\varphi \in H^{-1/2}(\Gamma; \mathbb{C}^2)$  \begin{equation*}
    T^D_\pm \Phi_z^{m, c} \varphi = \mp \frac{i}{2 c} (\sigma \cdot \nu) \varphi + \mathcal{C}_z^{m, c} \varphi.
  \end{equation*}
\end{lem}

Eventually, we note that Lemma~\ref{lemma_Plemelj_Sokhotskii} allows to show, in the same way as in \cite[Lemma~3.3~(ii)]{AMV14}, 
the identity
\begin{equation} \label{C_z_inv}
  -4 c^2 \big( (\sigma \cdot \nu) \mathcal{C}_z^{m, c} \big)^2 = -4 c^2 \big(  \mathcal{C}_z^{m, c} (\sigma \cdot \nu) \big)^2 = I,\quad z \in \rho(A_0^{m, c}).
\end{equation}

\subsection{A boundary triple for $A_{\eta, \tau, \lambda}^{m, c}$} \label{section_bt}

In this section we show how the operator $A_{\eta, \tau, \lambda}^{m, c}$ defined in~\eqref{def_A_eta_tau_intro} can be described with the help of a boundary triple. The general definition and some useful properties of boundary triples are recalled in Appendix~\ref{appendix_bt}. As before, $\Gamma$, $\Omega_\pm$, and $\nu$ are as in Hypothesis~\ref{hypothesis_Gamma}, and we use for $f: \mathbb{R}^2 \rightarrow \mathbb{C}^2$ the notation $f_\pm := f \upharpoonright \Omega_\pm$. First, we introduce an auxiliary Dirac operator and discuss its properties. Recall that $A_0^{m, c}$ is the free Dirac operator given by~\eqref{def_free_op}, define the unitary operator
\begin{equation*}
  U: L^2(\mathbb{R}^2; \mathbb{C}^2) \rightarrow L^2(\mathbb{R}^2; \mathbb{C}^2), \quad U(f_+ \oplus f_-) = f_+ \oplus (-f_-),
\end{equation*}
and introduce in $L^2(\mathbb{R}^2; \mathbb{C}^2)$ the operator
\begin{equation} \label{def_B}
  B := U^* A_0^{m, c} U.
\end{equation}
In a more explicit form the operator $B$ is given by
\begin{equation*}
  \begin{split}
    B f &= \big( -i c (\sigma \cdot \nabla) + m c^2 \sigma_3 \big) f_+ \oplus \big( -i c (\sigma \cdot \nabla) + m c^2 \sigma_3 \big) f_-, \\
    \dom B &= \big\{ f = f_+ \oplus f_- \in H^1(\Omega_+; \mathbb{C}^2) \oplus H^1(\Omega_-; \mathbb{C}^2): T^D_+ f_+ + T^D_- f_- = 0 \big\}.
  \end{split}
\end{equation*}
The basic properties of $B$ are stated in the following lemma. Recall that $\widetilde{\Phi}_z^{m, c}$ and $\mathcal{C}_z^{m, c}$ are the operators defined by~\eqref{mapping_properties_Phi_z} and~\eqref{def_C_z_formal}, respectively.

\begin{lem} \label{lemma_B}
  The operator $B$ is self-adjoint in $L^2(\mathbb{R}^2; \mathbb{C}^2)$, its spectrum is given by $\sigma(B) = \sigma_\textup{ess}(B) = \sigma(A_0^{m, c}) = (-\infty, -|m| c^2] \cup [|m| c^2, +\infty)$, and for $z \in \rho(B)$ one has 
  \begin{equation*}
    (B - z)^{-1} = (A_0^{m, c} - z)^{-1} + 4 c^2 \widetilde{\Phi}_z^{m, c} (\sigma \cdot \nu) \mathcal{C}_z^{m, c} (\sigma \cdot \nu) (\Phi_{\overline{z}}^{m, c})^*.
  \end{equation*}
\end{lem}
\begin{proof}
  Since the operator $U$ is unitary, it follows from the definition of $B$ in~\eqref{def_B} and the properties of $A_0^{m, c}$ in~\eqref{def_free_op} that $B$ is self-adjoint 
  in $L^2(\mathbb{R}^2; \mathbb{C}^2)$ and that $\sigma(B) = \sigma(A_0^{m, c})$. It remains to show the claimed resolvent formula. For this, fix $z \in \rho(B) = \rho(A_0^{m, c})$ and $f \in L^2(\mathbb{R}^2; \mathbb{C}^2)$, and define
  \begin{equation*}
    g := (A_0^{m, c} - z)^{-1}f + 4 c^2 \widetilde{\Phi}_z^{m, c} (\sigma \cdot \nu) \mathcal{C}_z^{m, c} (\sigma \cdot \nu) (\Phi_{\overline{z}}^{m, c})^* f.
  \end{equation*}
  Then, by the mapping properties of $(A_0^{m, c} - z)^{-1}$, $\widetilde{\Phi}_z^{m, c}$, $\mathcal{C}_z^{m, c}$, and $(\Phi_{\overline{z}}^{m, c})^*$ in~\eqref{mapping_properties_Phi_z}, \eqref{def_C_z_formal}, and~\eqref{Phi_star}, respectively, we get $g \in H^1(\Omega_+; \mathbb{C}^2) \oplus H^1(\Omega_-; \mathbb{C}^2)$. Moreover, using Lemma~\ref{lemma_Plemelj_Sokhotskii} and~\eqref{C_z_inv}, we find 
  \begin{equation*}
    \begin{split}
      T^D_+ g_+ + T^D_- g_- &= 2 T^D (A_0^{m, c} - z)^{-1} f + 8 c^2 \mathcal{C}_z^{m, c} (\sigma \cdot \nu) \mathcal{C}_z^{m, c} (\sigma \cdot \nu) (\Phi_{\overline{z}}^{m, c})^* f \\
      &= 2 T^D (A_0^{m, c} - z)^{-1} f - 2 (\Phi_{\overline{z}}^{m, c})^* f = 0
    \end{split}
  \end{equation*}
  and hence, $g \in \dom B$. Finally, equation~\eqref{equation_Phi_z_solution} yields
  \begin{equation*}
    ((B-z) g)_\pm =  (-i c (\sigma \cdot \nabla) + m c^2 \sigma_3 - z \sigma_0) \big((A_0^{m, c} - z)^{-1} f \big)_\pm = f_\pm.
  \end{equation*}
  Taking the injectivity of $B-z$ and the definition of $g$ into account, we see that also the claimed resolvent formula is true.
\end{proof}

Next, we show how a boundary triple that is useful to study $A_{\eta, \tau, \lambda}^{m, c}$ is given, see also \cite[Theorem~4.8]{BHSS22} and \cite{BHOP20, CLMT21} for similar constructions. In order to formulate the result, recall the definition of the map $\Lambda$ in~\eqref{def_Lambda}, of the free Dirac operator $A_0^{m, c}$ in~\eqref{def_free_op}, of its restriction $S$ and extension $S^*$ in~\eqref{def_S} and~\eqref{def_S_star}, respectively, of the boundary integral operator $\mathcal{C}_z^{m, c}$ in~\eqref{def_C_z_formal}, and of the unitary map $V$ in~\eqref{def_V}.

\begin{prop} \label{proposition_boundary_triple}
  Let $\zeta \in \rho(A_0^{m, c})$. Define the maps $\Gamma_0, \Gamma_1: \dom S^* \rightarrow L^2(\Gamma; \mathbb{C}^2)$ acting on $f = f_+ \oplus f_- \in \dom S^*$ by
  \begin{equation*}
    \begin{split}
      \Gamma_0 f &:= \frac{1}{2} \Lambda^{-1} V \big( T^D_+ f_+ + T^D_- f_- \big), \\
      \Gamma_1 f &:= -\Lambda \big( i c V (\sigma \cdot \nu) (T^D_+ f_+ - T^D_- f_-) \\
      &\qquad \qquad +c^2 V (\sigma \cdot \nu) (\mathcal{C}_\zeta^{m, c} + \mathcal{C}_{\overline{\zeta}}^{m, c}) (\sigma \cdot \nu) (T^D_+ f_+ + T^D_- f_-) \big). 
    \end{split}
  \end{equation*}
  Then $\{ L^2(\Gamma; \mathbb{C}^2), \Gamma_0, \Gamma_1 \}$ is a boundary triple for $S^*$. Moreover, the associated $\gamma$-field and Weyl function are given by
  \begin{equation*}
    \rho(A_0^{m, c}) \ni z \mapsto \gamma(z) = -4 c^2 \Phi_z^{m, c} (\sigma \cdot \nu) \mathcal{C}_z^{m, c} (\sigma \cdot \nu) V^* \Lambda
  \end{equation*}
  and
  \begin{equation*}
    \rho(A_0^{m, c}) \ni z \mapsto M(z) = 4 c^2 \Lambda V (\sigma \cdot \nu) \left( \mathcal{C}_z^{m, c} - \frac{1}{2} (\mathcal{C}_\zeta^{m, c} + \mathcal{C}_{\overline{\zeta}}^{m, c}) \right) (\sigma \cdot \nu) V^* \Lambda,
  \end{equation*}
respectively.  
\end{prop}
\begin{proof}
  We apply the construction from the end of Appendix~\ref{appendix_bt} leading to Proposition~\ref{proposition_boundary_triple_singular_perturbation} with the self-adjoint operator $B$ in \eqref{def_B} and the boundary mapping $\mathcal{T}: \dom B \rightarrow L^2(\Gamma; \mathbb{C}^2)$ given by 
  \begin{equation}\label{ttt}
    \mathcal{T} f = -i c \Lambda V (\sigma \cdot \nu) (T^D_+ f_+ - T^D_- f_-), \quad f \in \dom B.
  \end{equation}
  Since $T^D_\pm: H^1(\Omega_\pm; \mathbb{C}^2) \rightarrow H^{1/2}(\Gamma; \mathbb{C}^2)$ is surjective, $V$ and $\sigma \cdot \nu$ give rise to bijective maps in $H^{1/2}(\Gamma; \mathbb{C}^2)$ (these functions are pointwise invertible and $C^\infty$-smooth), and $\Lambda: H^{1/2}(\Gamma; \mathbb{C}^2) \rightarrow L^2(\Gamma; \mathbb{C}^2)$ is bijective by construction~\eqref{def_Lambda}, the operator $\mathcal{T}$ is surjective. Moreover, as $\dom B$ is a closed subspace of $H^1(\mathbb{R}^2 \setminus \Gamma; \mathbb{C}^2)$, the embedding of $\dom B$ equipped with the norm inherited from $H^1(\mathbb{R}^2 \setminus \Gamma; \mathbb{C}^2)$ into $\dom B$ equipped with the graph norm is bounded and bijective. Therefore,
  the graph norm induced by $B$ and the norm in $H^1(\mathbb{R}^2 \setminus \Gamma; \mathbb{C}^2)$ are equivalent on $\dom B$. Hence, by the properties of $T^D_\pm$, $\sigma \cdot \nu$, $V$, and $\Lambda$ the map $\mathcal{T}$ is continuous. Furthermore, $\ker \mathcal{T} = H^1_0(\mathbb{R}^2 \setminus \Sigma; \mathbb{C}^2)$ is dense in $L^2(\mathbb{R}^2; \mathbb{C}^2)$. Next, with Lemma~\ref{lemma_Plemelj_Sokhotskii} and Lemma~\ref{lemma_B}, one finds 
  \begin{equation}\label{ggg}
    \mathcal{G}_z = (\mathcal{T} (B - \overline{z})^{-1})^* = -4 c^2 \Phi_z^{m, c} (\sigma \cdot \nu) \mathcal{C}_z^{m, c} (\sigma \cdot \nu) V^* \Lambda
  \end{equation}
  for $z\in\rho(A_0^{m, c})=\rho(B)$. Recall that $\zeta \in \rho(B)$.
  For $f = f_\zeta + \mathcal{G}_\zeta\xi$ with $f_\zeta\in\dom B$ and $\xi \in L^2(\Gamma; \mathbb{C}^2)$ it follows from \eqref{ggg}, 
  Lemma~\ref{lemma_Plemelj_Sokhotskii}, and~\eqref{C_z_inv} that
  \begin{equation}
  \begin{split}
  &\frac{1}{2} \Lambda^{-1} V \big( T^D_+  + T^D_-  \big)(f_\zeta + \mathcal{G}_\zeta \xi)\\
  &\qquad\qquad =-2c^2 \Lambda^{-1} V \big( T^D_+  + T^D_-  \big)\Phi_\zeta^{m, c} (\sigma \cdot \nu) \mathcal{C}_\zeta^{m,c}(\sigma \cdot \nu) V^* \Lambda\xi\\
  &\qquad\qquad =-4c^2 \Lambda^{-1} V \big(\mathcal{C}_\zeta^{m,c}(\sigma \cdot \nu)\big)^2 V^* \Lambda\xi\\
  &\qquad \qquad=\xi,
  \end{split}
  \end{equation}
  and hence we obtain the representation of $\Gamma_0$ from Proposition~\ref{proposition_boundary_triple_singular_perturbation}. In order to 
  obtain the formula for $\Gamma_1$ we note first that the mapping $\mathcal T$ in \eqref{ttt} can be extended onto $\dom S^*$. Using 
  \eqref{ttt}, \eqref{ggg}, and Lemma~\ref{lemma_Plemelj_Sokhotskii}  we then compute
  \begin{equation}\label{soistes}
  \begin{split}
   \mathcal T\mathcal G_z\xi&=4ic^3\Lambda V (\sigma \cdot \nu) (T^D_+  - T^D_-)\Phi_z^{m, c} (\sigma \cdot \nu) \mathcal{C}_z^{m, c} (\sigma \cdot \nu) V^* \Lambda\xi\\
   &=4c^2\Lambda V (\sigma \cdot \nu)  \mathcal{C}_z^{m, c} (\sigma \cdot \nu) V^* \Lambda\xi
   \end{split}
  \end{equation}
  for $z\in\rho(A_0^{m, c})=\rho(B)$ and $\xi \in L^2(\Gamma; \mathbb{C}^2)$. Now
  consider $f = f_\zeta + \mathcal{G}_\zeta\xi= f_{\overline\zeta} + \mathcal{G}_{\overline\zeta}\xi$ and observe that 
  by Proposition~\ref{proposition_boundary_triple_singular_perturbation}
  \begin{equation}\label{gamma1}
   \Gamma_1 f=\frac{1}{2}\mathcal T (f_\zeta+f_{\overline\zeta})=\frac{1}{2}\mathcal T (f-\mathcal{G}_\zeta \xi + f-\mathcal{G}_{\overline\zeta} \xi)=\mathcal T f
   -\frac{1}{2}\mathcal T \big(\mathcal{G}_\zeta + \mathcal{G}_{\overline\zeta}\big) \xi.
  \end{equation}
  Inserting $\mathcal Tf$ and the expressions $\mathcal T\mathcal G_\zeta\xi$, $\mathcal T\mathcal G_{\overline\zeta}\xi$ from \eqref{soistes} in \eqref{gamma1}, and making use of 
  $\xi=\Gamma_0 f= \frac{1}{2} \Lambda^{-1} V( T^D_+ f_+ + T^D_- f_- )$ we obtain the representation of $\Gamma_1$. 
  Finally, the expressions for the $\gamma$-field and Weyl function follow immediately from $\gamma(z)=\mathcal G_z$ and $M(z)=\mathcal T(\mathcal G_z-\tfrac{1}{2}(\mathcal G_\zeta+
  \mathcal G_{\overline\zeta}))$ in Proposition~\ref{proposition_boundary_triple_singular_perturbation}.
\end{proof}

In the following proposition we describe the operator $A_{\eta, \tau, \lambda}^{m, c}$ with the help of the boundary triple from the previous proposition. For similar arguments see \cite[Section~6]{BHSS22} and also \cite[Proposition~6.3]{CLMT21} or \cite[Proposition~4.3]{BHOP20}. Recall that $V$ is given by~\eqref{def_V} and define for $\eta, \tau, \lambda \in \mathbb{R}$ the matrix
\begin{equation} \label{def_B_matrix}
  F := \begin{pmatrix} \eta + \tau & \lambda \\ \lambda & \eta - \tau \end{pmatrix},
\end{equation}
implying 
\begin{equation} \label{eq:VF_identity}
V^* F V = \eta \sigma_0 + \tau \sigma_3 + i \lambda (\sigma \cdot \nu) \sigma_3.
\end{equation}

\begin{prop} \label{proposition_A_eta_tau_bt}
  Let $\eta, \tau, \lambda \in \mathbb{R}$ and define in $L^2(\Gamma; \mathbb{C}^2)$ the operator
  \begin{equation} \label{def_Theta1}
    \begin{split}
      \Theta \varphi &:= \Lambda \big[ F - 2 c^2 V (\sigma \cdot \nu) (\mathcal{C}_\zeta^{m, c} + \mathcal{C}_{\overline{\zeta}}^{m, c}) (\sigma \cdot \nu) V^* \big] \Lambda \varphi, \\
      \dom \Theta &= \big \{ \varphi \in L^2(\Gamma; \mathbb{C}^2): \\
      &\qquad \quad \big[ F - 2 c^2 V (\sigma \cdot \nu) (\mathcal{C}_\zeta^{m, c} + \mathcal{C}_{\overline{\zeta}}^{m, c}) ( \sigma \cdot \nu) V^* \big] \Lambda \varphi \in H^{1/2}(\Gamma; \mathbb{C}^2) \big\}.
    \end{split}
  \end{equation}
  Then $A_{\eta, \tau, \lambda}^{m, c} = S^* \upharpoonright \ker(\Gamma_1 - \Theta \Gamma_0)$.
\end{prop}
\begin{proof}
  Recall that $f \in \dom A_{\eta, \tau, \lambda}^{m, c}$ if and only if $f_\pm \in H(\sigma, \Omega_\pm)$ and
  \begin{equation*}
    -i c (\sigma \cdot \nu) (T^D_+ f_+ - T^D_- f_-) = \frac{1}{2} (\eta \sigma_0 + \tau \sigma_3 + i \lambda (\sigma \cdot \nu) \sigma_3) (T^D_+ f_+ + T^D_- f_-).
  \end{equation*}
  Using \eqref{eq:VF_identity} and 
  \begin{equation*}
    \begin{split}
      \frac{1}{2} (T^D_+ f_+ + T^D_- f_-) &= V^* \Lambda \Gamma_0 f, \\
      -i c (\sigma \cdot \nu) (T^D_+ f_+ - T^D_- f_-) &= V^* \Lambda^{-1} \Gamma_1 f +  2 c^2 (\sigma \cdot \nu) (\mathcal{C}_\zeta^{m, c} + \mathcal{C}_{\overline{\zeta}}^{m, c}) (\sigma \cdot \nu) V^* \Lambda \Gamma_0 f,
    \end{split}
  \end{equation*}
  one finds that $f \in \dom A_{\eta, \tau, \lambda}^{m, c}$ if and only if $f \in \dom S^*$ and $\Gamma_1 f - \Theta \Gamma_0 f  =0$.
\end{proof}

\section{Proof of Theorem~\ref{theorem_intro} and further properties of $A_{\eta, \tau, \lambda}^{m, c}$} \label{section_proof}

The claims in Theorem~\ref{theorem_intro} are shown in various steps: Subsection~\ref{section_self_adjoint} is devoted to the self-adjointness of $A_{\eta, \tau, \lambda}^{m, c}$, the claim in Theorem~\ref{theorem_intro}~(i) about the Sobolev regularity of $\dom A_{\eta, \tau, \lambda}^{m, c}$ is then shown in Subsection~\ref{section_domain_regularity}. Afterwards, $\sigma_\textup{ess}(A_{\eta, \tau, \lambda}^{m, c})$ is computed in Subsection~\ref{section_essential_spectrum}. Finally, in Subsection~\ref{section_krein} we state some further properties of $A_{\eta, \tau, \lambda}^{m, c}$ like a useful Krein type resolvent formula and a Birman-Schwinger principle.

We will make use of the following two main ingredients: We use that $A_{\eta, \tau, \lambda}^{m, c}$ can be described with the boundary triple in 
Proposition~\ref{proposition_boundary_triple} as explained in Proposition~\ref{proposition_A_eta_tau_bt}, and the abstract results in Appendix~\ref{appendix_bt}. And we will use that the claimed properties can be obtained explicitly when the interaction support is the straight line $\Sigma$; cf. \cite{BHT23}. To distinguish the appearing objects on the straight line from the case of the more general interaction support $\Gamma$, we will denote them with a superindex $\Sigma$. Eventually, we will make often  use of the arc-length parametrization
\begin{equation} \label{arc_length_straight_line}
  x: \mathbb{R} \rightarrow \mathbb{R}^2, \quad x(t) = (t, 0),
\end{equation}
of the straight line $\Sigma$.

\subsection{Self-adjointness of $A_{\eta, \tau, \lambda}^{m, c}$} \label{section_self_adjoint}

Let $\{ L^2(\Gamma; \mathbb{C}^2), \Gamma_0, \Gamma_1 \}$ be the boundary triple for $S^*$ from Proposition~\ref{proposition_boundary_triple}  and $\Theta$ be defined by~\eqref{def_Theta1}. Recall that $$A_{\eta, \tau, \lambda}^{m, c} = S^* \upharpoonright \ker (\Gamma_1 - \Theta \Gamma_0).$$ 
Hence, by Theorem~\ref{theorem_bt} the operator $A_{\eta, \tau, \lambda}^{m, c}$ is self-adjoint in $L^2(\mathbb{R}^2; \mathbb{C}^2)$ if and only if $\Theta$ is self-adjoint in $L^2(\Gamma; \mathbb{C}^2)$. We show that the latter is the case. 

Denote by $\Lambda^\Sigma$ the map in~\eqref{def_Lambda} on $\Sigma$, by $\mathcal{C}_z^{m, c, \Sigma}$ the operator in~\eqref{def_C_z_formal} on the straight line $\Sigma$, and we write $\Theta^\Sigma$ for the parameter defined as in~\eqref{def_Theta1} on $\Sigma$. Since $\sigma \cdot \nu^\Sigma = -\sigma_2$ and $V^\Sigma = \sigma_0$ on $\Sigma$, one has the more explicit representation
\begin{equation} \label{Theta_Sigma}
  \Theta^\Sigma = \Lambda^\Sigma \big[ F - 2 c^2 \sigma_2 (\mathcal{C}_\zeta^{m, c, \Sigma} + \mathcal{C}_{\overline{\zeta}}^{m, c, \Sigma}) \sigma_2 \big] \Lambda^\Sigma
\end{equation}
and as in~\eqref{def_Theta1} this map  is defined on the maximal domain. This operator is self-adjoint in $L^2(\Sigma; \mathbb{C}^2)$ by Theorem~\ref{theorem_bt} and Proposition~\ref{proposition_A_eta_tau_bt} (applied for $\Gamma = \Sigma$), as $A_{\eta, \tau, \lambda}^{m, c, \Sigma}$ is self-adjoint in $L^2(\mathbb{R}^2; \mathbb{C}^2)$; this follows for $d \neq -4 c^2$ from \cite[Section~6]{BHT23} noting that $A_{\eta, \tau, \lambda}^{m,c, \Sigma} = c A_{\eta/c, \tau/c, \lambda/c}^{mc, 1, \Sigma}$ and $A_{\eta/c, \tau/c, \lambda/c}^{mc, 1, \Sigma} = A_{\eta/c, \tau/c, \lambda/c}$ is the operator studied in \cite{BHT23}, when $m$ is replaced by $mc$, and for $d = -4c^2$ from \cite{GL06}.
Next, recall that $\gamma$ is the arc-length parametrization of $\Gamma$ given by~\eqref{arc_length_Gamma_intro} and that $x$ is the arc-length parametrization of $\Sigma$ in~\eqref{arc_length_straight_line}. By definition, one has $\Lambda^\Sigma = U_x^{-1} U_\gamma \Lambda U_\gamma^{-1} U_x$. Then, by Proposition~\ref{proposition_C_z_Gamma} 
\begin{equation} \label{difference_Theta_Gamma_Sigma} 
  \begin{split}
    U_\gamma &\Theta U_\gamma^{-1} - U_x \Theta^\Sigma U_x^{-1} \\
    &= -2 c^2 U_\gamma \Lambda U_\gamma^{-1} \big[ U_\gamma V (\sigma \cdot \nu) \mathcal{C}_\zeta^{m, c} (\sigma \cdot \nu) V^* U_\gamma^{-1} - U_x \sigma_2 \mathcal{C}_\zeta^{m, c, \Sigma} \sigma_2 U_x^{-1} \\
    &\qquad \qquad \quad + U_\gamma V (\sigma \cdot \nu) \mathcal{C}_{\overline{\zeta}}^{m, c} (\sigma \cdot \nu) V^* U_\gamma^{-1} - U_x \sigma_2 \mathcal{C}_{\overline{\zeta}}^{m, c, \Sigma} \sigma_2 U_x^{-1} \big] U_\gamma \Lambda U_\gamma^{-1},
  \end{split}
\end{equation}
defined on its natural maximal domain,
can be extended to a compact operator in $L^2(\mathbb{R}; \mathbb{C}^2)$. Since the adjoint of $\mathcal{C}_\zeta^{m, c}$ in $L^2(\Gamma; \mathbb{C}^2)$ is given by $\mathcal{C}_{\overline{\zeta}}^{m, c}$, the right hand side of the latter equation is symmetric and can be extended to a self-adjoint operator. As $\Theta^\Sigma$ is self-adjoint in $L^2(\Sigma; \mathbb{C}^2)$ and $U_\gamma: L^2(\Gamma; \mathbb{C}^2) \rightarrow L^2(\mathbb{R}; \mathbb{C}^2)$ and $U_x: L^2(\Sigma; \mathbb{C}^2) \rightarrow L^2(\mathbb{R}; \mathbb{C}^2)$ are unitary, this implies that $\Theta$ is self-adjoint.

\subsection{Sobolev regularity in $\dom A_{\eta, \tau, \lambda}^{m, c}$} \label{section_domain_regularity}

In this subsection we show item~(i) of Theorem~\ref{theorem_intro}. For this we prove an auxiliary lemma in which we analyze $\dom \Theta$. Afterwards the Krein type resolvent formula for $A_{\eta, \tau, \lambda}^{m, c}$ from Theorem~\ref{theorem_bt} allows us to show the claims on $\dom A_{\eta, \tau, \lambda}^{m, c}$. Since $m c^2 \sigma_3$ is a bounded perturbation of the self-adjoint operator $A_{\eta, \tau, \lambda}^{m, c}$, it does not affect $\dom A_{\eta, \tau, \lambda}^{m, c}$, and hence, it is no restriction to assume in this subsection that $m \neq 0$ and to choose $\zeta = 0$ in Proposition~\ref{proposition_boundary_triple}.

\begin{lem} \label{lemma_dom_Theta}
  Assume that $m \neq 0$. Then, for the domain of the operator $\Theta$ defined in~\eqref{def_Theta1} with $\zeta = 0$ one has:
  \begin{itemize}
    \item[(a)] If $(\tfrac{d}{4} - c^2)^2 - \lambda^2 c^2 \neq 0$, then $\dom \Theta = H^1(\Gamma; \mathbb{C}^2)$.
    \item[(b)] If $(\tfrac{d}{4} - c^2)^2 - \lambda^2 c^2 = 0$, then $\dom \Theta \not\subset H^s(\Gamma; \mathbb{C}^2)$ for all $s > 0$.
  \end{itemize}
\end{lem}
\begin{proof}
  The proof of this lemma is split into two steps: We analyze $\dom \Theta^\Sigma$ first and with the help of Proposition~\ref{proposition_C_z_Gamma} we transfer
  these results to the case of more general curves $\Gamma$.

  \textit{Step~1}.
  We use the Fourier transform $\mathcal{F}$ in $L^2(\mathbb{R}; \mathbb{C}^2) = U_x L^2(\Sigma; \mathbb{C}^2)$ and note first that $\mathcal{F} U_x \mathcal{C}_z^{m, c, \Sigma} U_x^{-1} \mathcal{F}^{-1}$ is the maximal multiplication operator associated with the function
  \begin{equation*}
    c_z(p) = \frac{1}{2 \sqrt{p^2 c^2 + (mc^2)^2 - z^2}} \begin{pmatrix} \tfrac{z}{c} + m c & p \\ p & \tfrac{z}{c} - m c \end{pmatrix};
  \end{equation*}
  to see this, one can apply \cite[Lemma~2.1]{BHT22} and note with~\eqref{relation_G_z_c} that $\mathcal{C}_z^{m, c, \Sigma} = \tfrac{1}{c} \mathcal{C}_{z/c}^{m c, 1, \Sigma}$ and $\mathcal{C}_z = \mathcal{C}_{z}^{mc, 1, \Sigma}$ in the notation of \cite{BHT22}. Recall that we chose $\zeta = 0$ in the definition of $\Theta^\Sigma$. 
  Hence, with~\eqref{Theta_Sigma} one sees that the operator $\mathcal{F} U_x \Theta^\Sigma U_x^{-1} \mathcal{F}^{-1}$ is the maximal multiplication operator associated with the function
  \begin{equation*}
    \theta(p) = \sqrt{p^2 + 1} \left[ F - \sigma_2 \frac{2 c^2}{\sqrt{p^2 c^2 + (mc^2)^2}} \big( m c \sigma_3 + p \sigma_1 \big) \sigma_2 \right] = \theta_1(p) + \theta_2(p)
  \end{equation*}
  with 
  \begin{equation*}
    \begin{split}
      \theta_1(p) &= \sqrt{p^2 + 1} \begin{pmatrix} \eta + \tau & \lambda + \frac{2 p c}{\sqrt{p^2 + (m c)^2}} \\ \lambda + \frac{2 p c}{\sqrt{p^2 + (m c)^2}} & \eta - \tau \end{pmatrix}, \\ 
      \theta_2(p) &= \frac{2 \sqrt{p^2 + 1}}{\sqrt{p^2 + (mc)^2}} m c^2 \sigma_3.
    \end{split}
  \end{equation*}
  Since $\theta_2$ is a bounded function, its associated multiplication operator is bounded and everywhere defined in $L^2(\mathbb{R}; \mathbb{C}^2)$. The eigenvalues of $\theta_1(p)$ are given by 
  \begin{equation*}
    \begin{split}
      \mu_{\pm}(p) &= \sqrt{p^2 + 1} \left( \eta \pm \sqrt{\eta^2 - d + \frac{4 p^2 c^2}{p^2 + (m c)^2} + \frac{4 p c \lambda}{\sqrt{p^2 + (m c)^2}}} \right) \\
      &= \sqrt{p^2 + 1} \left( \eta \pm \sqrt{\eta^2 - d + 4 c^2 + 4 \lambda c \, \sign(p) + f(p)} \right),
    \end{split}
  \end{equation*}
  where $d$ is given by~\eqref{eq:d_def} and  $f$ is a continuous function  such  that $f(p)=\mathcal{O} \big( \frac{1}{p^2} \big)$ for $p \rightarrow \pm \infty$.
  Hence, if $(\tfrac{d}{4} - c^2)^2 - \lambda^2 c^2 \neq 0$, then $\mu_{\pm}(p) \sim \sqrt{p^2 + 1}$ and thus, $\dom \Theta^\Sigma = H^1(\Sigma; \mathbb{C}^2)$. On the other hand, if $(\tfrac{d}{4} - c^2)^2 - \lambda^2 c^2 = 0$, then at least one of the functions $\mu_{\pm}$ is bounded on $\mathbb{R}_+$ or on $\mathbb{R}_-$ (or bounded everywhere on $\mathbb{R}$, if $0 = \lambda = d - 4c^2$), implying that $\dom \Theta^\Sigma \not\subset H^s(\Sigma; \mathbb{C}^2)$ for all $s > 0$. Thus, the claim of the lemma is true for the straight line $\Gamma = \Sigma$.
  
  \textit{Step~2}.
  Recall that $U_\gamma \Theta U_\gamma^{-1} - U_x \Theta^\Sigma U_x^{-1}$ can be extended to a compact operator in $L^2(\mathbb{R}; \mathbb{C})$, see Proposition~\ref{proposition_C_z_Gamma} and \eqref{difference_Theta_Gamma_Sigma}. Hence, the maximal domain of $\Theta$ can be described in the form $\dom \Theta =  \dom (\Theta^\Sigma U_x^{-1} U_\gamma)$
  and now the claim follows from our observation in \textit{Step~1}.
\end{proof}

\begin{proof}[Proof of Theorem~\ref{theorem_intro}~(i)]
  As explained above we can assume, without loss of generality, that $m \neq 0$, and choose $\zeta = 0$ in~\eqref{def_Theta1}. To prove item~(a) we use $\dom A_{\eta, \tau, \lambda}^{m, c} = \ran (A_{\eta, \tau, \lambda}^{m, c} - z)^{-1}$ for $z \in \mathbb{C} \setminus \mathbb{R}$ and the resolvent formula
  \begin{equation} \label{krein}
    (A_{\eta, \tau, \lambda}^{m, c} - z)^{-1} = (A_0^{m, c} - z)^{-1} + \gamma(z) \big( \Theta - M(z) \big)^{-1} \gamma(\overline{z})^*
  \end{equation}
  from Theorem~\ref{theorem_bt}~(iv). If $(\tfrac{d}{4} - c^2)^2 - \lambda^2 c^2 \neq 0$, then by Lemma~\ref{lemma_dom_Theta} one has 
  $$\dom (\Theta - M(z)) = \dom \Theta = H^1(\Gamma; \mathbb{C}^2).$$ 
  Hence, $\ran(A_0^{m, c}-z)^{-1}=\dom A_0^{m, c} = H^1(\mathbb{R}^2; \mathbb{C}^2)$, the explicit form $\gamma(z) = -4 c^2 \Phi_z^{m, c} (\sigma \cdot \nu) \mathcal{C}_z^{m, c} (\sigma \cdot \nu) V^* \Lambda$ from Proposition~\ref{proposition_boundary_triple}, and the mapping properties in~\eqref{def_Lambda}, \eqref{mapping_properties_Phi_z}, and~\eqref{def_C_z_formal} imply  $\dom A_{\eta, \tau, \lambda}^{m, c} \subset H^1(\mathbb{R}^2 \setminus \Gamma; \mathbb{C}^2)$, that is, the assertion in Theorem~\ref{theorem_intro}~(i)~(a) holds.
  
  Next, we show assertion~(b), i.e. that $\dom A_{\eta, \tau, \lambda}^{m, c} \not\subset H^s(\mathbb{R}^2 \setminus \Gamma; \mathbb{C}^2)$ for all $s>0$, if  $(\tfrac{d}{4} - c^2)^2 - \lambda^2 c^2 = 0$. Note that, due to the surjectivity of $(\Gamma_0, \Gamma_1)$, for all $\varphi \in \dom \Theta$ there exists $f \in \dom S^*$ such that $\Gamma_0 f = \varphi$ and $\Gamma_1 f = \Theta \varphi$. Hence, for each $\varphi \in \dom \Theta$ there exists $f \in \dom A_{\eta, \tau, \lambda}^{m, c}$ such that $\Gamma_0 f = \varphi$. If we assume now $\dom A_{\eta, \tau, \lambda}^{m, c} \subset H^s(\mathbb{R}^2 \setminus \Gamma; \mathbb{C}^2)$ for an $s > 0$, then by the definition of $\Gamma_0$ we would conclude that any $\varphi \in \dom \Theta$ also belongs to $H^s(\Gamma; \mathbb{C}^2)$. However, this contradicts Lemma~\ref{lemma_dom_Theta}, showing that also Theorem~\ref{theorem_intro}~(i)~(b) is true. 
\end{proof}

\subsection{Essential spectrum of $A_{\eta, \tau, \lambda}^{m, c}$} \label{section_essential_spectrum}

First, we claim that for all combinations $\eta, \tau, \lambda \in \mathbb{R}$ we have 
\begin{equation}\label{showthisplease}
(-\infty, -|m|c^2] \cup [|m|c^2, +\infty) \subset \sigma_\textup{ess}(A_{\eta, \tau, \lambda}^{m, c}),
\end{equation}
which can be proved via a suitable singular sequence. To show \eqref{showthisplease}, we can assume that $\Gamma \subset [0, +\infty) \times \mathbb{R}$. In fact, if
this is not the case, then via a suitable rotation and translation one can transform $A_{\eta, \tau, \lambda}^{m, c}$ to a unitary equivalent Dirac operator, which may have different Dirac matrices $\alpha_1, \alpha_2, \beta \in \mathbb{C}^{2 \times 2}$ instead of $\sigma_1, \sigma_2, \sigma_3$, see, e.g., \cite[p. 150]{J05} or \cite[Proposition~4]{R22}, with a singular potential that is supported on a curve 
contained in $[0, +\infty) \times \mathbb{R}$, so this assumption is not a restriction. 

Let  now $z \in (-\infty, -|m| c^2) \cup (|m| c^2, +\infty)$ be fixed and consider for $n \in \mathbb{N}$ the function
\begin{equation*} 
  f_n(x_1, x_2) := \frac{1}{n} \chi\left(\frac{1}{n} |x-y_n| \right) e^{i \sqrt{z^2 - m^2 c^4} x_1/c} \big(\sqrt{z^2 - m^2 c^4} \sigma_1 + m c^2 \sigma_3 + z \sigma_0\big) \xi,
\end{equation*}
where $\chi:\mathbb{R}\to [0, 1]$ is a $C^\infty$-function such that $\chi(t)=1$ for $|t|\leq \frac{1}{2}$ and $\chi(t)=0$ for $|t|\geq 1$, the vector $\xi \in \mathbb{C}^2$ is chosen such that $(\sqrt{z^2 - m^2 c^4} \sigma_1 + m c^2 \sigma_3 + z \sigma_0) \xi \neq 0$, 
and we set $y_n := (0, -n^2)$, $n \in \mathbb{N}$. 
Then $f_n \in \dom A_{\eta, \tau, \lambda}^{m, c}$ and one can show as in \cite[Theorem~5.7~(i)]{BH19} that $\|(A_{\eta, \tau, \lambda}^{m, c}-z)f_n\|/\|f_n\| \rightarrow 0$ for $n\to \infty$. Since $z \in (-\infty, -|m| c^2) \cup (|m| c^2, +\infty)$ was arbitrary, the inclusion \eqref{showthisplease} follows.
  
For $m=0$ the inclusion \eqref{showthisplease} implies $\sigma_\textup{ess}(A_{\eta, \tau, \lambda}^{0, c}) = \mathbb{R}$, so there is nothing left to show. Hence we can assume $m \neq 0$ in the following. In order to analyze $\sigma_\textup{ess}(A_{\eta, \tau, \lambda}^{m, c}) \cap (-|m| c^2, |m| c^2)$, we use that by Proposition~\ref{proposition_A_eta_tau_bt} and Theorem~\ref{theorem_bt}~(iii)
\begin{equation}\label{eins}
  z \in \sigma_\textup{ess}(A_{\eta, \tau, \lambda}^{m, c}) \cap (-|m| c^2, |m| c^2) \quad \text{if and only if} \quad 0 \in \sigma_\textup{ess}(\Theta - M(z)),
\end{equation}
where $M$ is the Weyl function of the boundary triple in Proposition~\ref{proposition_boundary_triple}. 
Using the definition of $\Theta$ in~\eqref{def_Theta1} 
one verifies for $z\in (-|m| c^2, |m| c^2)$
\begin{equation*}
  \Theta - M(z) = \Lambda [F - 4 c^2 V (\sigma \cdot \nu) \mathcal{C}_z^{m, c} (\sigma \cdot \nu) V^*] \Lambda
\end{equation*}
and in the same way for the straight line
\begin{equation*}
  \Theta^\Sigma - M^\Sigma(z) = \Lambda^\Sigma [F - 4 c^2 \sigma_2 \mathcal{C}_z^{m, c,\Sigma} \sigma_2)]\Lambda^\Sigma.
\end{equation*}
Since $\Lambda^\Sigma = U_x^{-1} U_\gamma \Lambda U_\gamma^{-1} U_x$ we obtain 
\begin{equation*} 
  \begin{split}
    U_\gamma &(\Theta - M(z) ) U_\gamma^{-1} - U_x (\Theta^\Sigma - M^\Sigma(z)) U_x^{-1} \\
    &= -4 c^2 U_\gamma \Lambda U_\gamma^{-1} \big[ U_\gamma V (\sigma \cdot \nu) \mathcal{C}_z^{m, c} (\sigma \cdot \nu) V^* U_\gamma^{-1} - U_x \sigma_2 \mathcal{C}_z^{m, c, \Sigma} \sigma_2 U_x^{-1} \big] U_\gamma \Lambda U_\gamma^{-1}
  \end{split}
\end{equation*}
and it follows from Proposition~\ref{proposition_C_z_Gamma} that this operator 
can be extended to a compact operator in $L^2(\mathbb{R}; \mathbb{C}^2)$. Since $U_\gamma$ and $U_x$ are unitary, we conclude that 
$$\sigma_\textup{ess}(\Theta - M(z) ) = \sigma_\textup{ess}(\Theta^\Sigma - M^\Sigma(z) ),$$ 
and hence, by \eqref{eins},
\begin{equation*}
  z \in \sigma_\textup{ess}(A_{\eta, \tau, \lambda}^{m, c}) \cap (-|m| c^2, |m| c^2) \quad \text{if and only if} \quad 0 \in \sigma_\textup{ess}(\Theta^\Sigma - M^\Sigma(z)).
\end{equation*}
Applying Theorem~\ref{theorem_bt}~(iii) and Proposition~\ref{proposition_A_eta_tau_bt} to the Dirac operator $A_{\eta, \tau, \lambda}^{m, c, \Sigma}$ with
the singular interaction supported on the straight line $\Sigma$, we find 
$\sigma_\textup{ess}(A_{\eta, \tau, \lambda}^{m, c}) = \sigma_\textup{ess}(A_{\eta, \tau, \lambda}^{m, c, \Sigma})$, which finally yields the statement in Theorem~\ref{theorem_intro}~(ii);
this follows in the case $\eta^2 - \tau^2 - \lambda^2 \neq -4 c^2$ from \cite[Theorem~6.2]{BHT23} noting that $A_{\eta, \tau, \lambda}^{m,c, \Sigma} = c A_{\eta/c, \tau/c, \lambda/c}^{mc, 1, \Sigma}$ and $A_{\eta/c, \tau/c, \lambda/c}^{mc, 1, \Sigma} = A_{\eta/c, \tau/c, \lambda/c}$ is the operator studied in \cite{BHT23}, when $m$ is replaced by $mc$; to obtain the result, it is convenient to replace the term $k$ in \cite[Theorem~6.2]{BHT23} by $k c$. The claim is also true for $\eta^2 - \tau^2 - \lambda^2 = -4 c^2$, which is excluded in \cite{BHT23}, see \cite{GL06}.

Finally, the claim that $\sigma(A_{\eta, \tau, \lambda}^{m, c}) = \sigma_\textup{ess}(A_{\eta, \tau, \lambda}^{m, c})$, if $\Gamma=\Sigma$ is the straight line, is a part of the results in \cite{BHT23, GL06}.

\subsection{Krein type resolvent formula and further spectral properties of $A_{\eta, \tau, \lambda}^{m, c}$} \label{section_krein}

In the next proposition we formulate a Birman-Schwinger principle to characterize discrete eigenvalues of $A_{\eta, \tau, \lambda}^{m, c}$. 
The proof of the following result is the same as the proof of \cite[Proposition~6.5~(ii)]{BHSS22} and will be omitted here.

\begin{prop} \label{proposition_Krein_formulaXXX}
  Let $\eta, \tau, \lambda \in \mathbb{R}$ and $z \notin \sigma_\textup{ess}(A_{\eta, \tau, \lambda}^{m, c})$. Then, 
  $$z \in \sigma_\textup{disc}(A_{\eta, \tau, \lambda}^{m, c})\quad\text{if and only if}\quad -1 \in \sigma_\textup{p}((\eta \sigma_0 + \tau \sigma_3 + i \lambda (\sigma \cdot \nu) \sigma_3) \mathcal{C}_z^{m,c}),$$
  where $(\eta \sigma_0 + \tau \sigma_3 + i \lambda (\sigma \cdot \nu) \sigma_3) \mathcal{C}_z^{m,c}$ is viewed as an operator in $H^{-1/2}(\Gamma; \mathbb{C}^2)$.
\end{prop}

Next, we state a Krein type resolvent formula for $A_{\eta, \tau, \lambda}^{m, c}$. Recall that the operators $A_0^{m, c}$ and $\Phi_z^{m, c}$ are given by~\eqref{def_free_op} and~\eqref{def_Phi_z}, respectively. 

\begin{prop} \label{proposition_Krein_formula}
  Let $\eta, \tau, \lambda \in \mathbb{R}$. Then, $z \in \rho(A_{\eta, \tau, \lambda}^{m, c})$ if and only if the operator $\sigma_0 + (\eta \sigma_0 + \tau \sigma_3 + i \lambda (\sigma \cdot \nu) \sigma_3) \mathcal{C}_z^{m,c}$ admits a bounded inverse from $H^{1/2}(\Gamma; \mathbb{C}^2)$ to $H^{-1/2}(\Gamma; \mathbb{C}^2)$ and in this case the resolvent formula
    \begin{equation*}
      \begin{split}
        (&A_{\eta, \tau, \lambda}^{m, c} - z)^{-1} = (A_0^{m, c} - z)^{-1} \\
        &~~- \Phi_z^{m, c} \big(\sigma_0 + (\eta \sigma_0 + \tau \sigma_3 + i \lambda (\sigma \cdot \nu) \sigma_3) \mathcal{C}_z^{m, c} \big)^{-1} (\eta \sigma_0 + \tau \sigma_3 + i \lambda (\sigma \cdot \nu) \sigma_3) (\Phi_{\overline{z}}^{m, c})^*
      \end{split}
    \end{equation*}
    holds.
\end{prop}
\begin{proof}
  First, one shows with Lemma~\ref{lemma_B} in exactly the same way as in \cite[Proposition~6.5~(i)]{BHSS22} that for $z \in \rho(A_{\eta, \tau, \lambda}^{m, c})$ the map $\sigma_0 + (\eta \sigma_0 + \tau \sigma_3 + i \lambda (\sigma \cdot \nu) \sigma_3) \mathcal{C}_z^{m,c}$ admits a bounded inverse from $H^{1/2}(\Gamma; \mathbb{C}^2)$ to $H^{-1/2}(\Gamma; \mathbb{C}^2)$ and that the claimed resolvent formula is true.
  
  Conversely, assume that $\sigma_0 + (\eta \sigma_0 + \tau \sigma_3 + i \lambda (\sigma \cdot \nu) \sigma_3) \mathcal{C}_z^{m,c}$ admits a bounded inverse from $H^{1/2}(\Gamma; \mathbb{C}^2)$ to $H^{-1/2}(\Gamma; \mathbb{C}^2)$. Then, by Proposition~\ref{proposition_Krein_formulaXXX} the operator $A_{\eta, \tau, \lambda}^{m,c} - z$ is injective. Furthermore, by the mapping properties of $\Phi_z^{m, c}$ and $(\Phi_z^{m, c})^*$ in~\eqref{def_Phi_z}, \eqref{equation_Phi_z_solution}, and~\eqref{Phi_star}, respectively, the map 
  \begin{equation*}
      \begin{split}
        &R_z := (A_0^{m, c} - z)^{-1} \\
        &~~- \Phi_z^{m, c} \big(\sigma_0 + (\eta \sigma_0 + \tau \sigma_3 + i \lambda (\sigma \cdot \nu) \sigma_3) \mathcal{C}_z^{m, c} \big)^{-1} (\eta \sigma_0 + \tau \sigma_3 + i \lambda (\sigma \cdot \nu) \sigma_3) (\Phi_{\overline{z}}^{m, c})^*
      \end{split}
    \end{equation*}
    is bounded in $L^2(\mathbb{R}^2; \mathbb{C}^2)$ and $\ran R_z \subset H(\sigma, \mathbb{R}^2 \setminus \Gamma)$. Moreover, with the help of Lemma~\ref{lemma_trace_theorem} one finds that for any $f \in L^2(\mathbb{R}^2; \mathbb{C}^2)$ the function $R_z f$ satisfies the transmission condition in $\dom A_{\eta, \tau, \lambda}^{m, c}$ and thus, $R_z f \in \dom A_{\eta, \tau, \lambda}^{m, c}$. Finally, one obtains with~\eqref{equation_Phi_z_solution} that
  \begin{equation*}
    ((A_{\eta, \tau, \lambda}^{m,c}-z) R_z f)_\pm =  (-i c (\sigma \cdot \nabla) + m c^2 \sigma_3 - z \sigma_0) \big((A_0^{m, c} - z)^{-1} f \big)_\pm = f_\pm.
  \end{equation*}
  Since $A_{\eta, \tau, \lambda}^{m,c} - z$ is injective, we see that also the claimed resolvent formula is true.
\end{proof}

Eventually, two symmetry relations about the eigenvalues of $A_{\eta, \tau, \lambda}^{m, c}$, that are sometimes referred to as isospectral relations, hold. For the proof of the following result, one can follow the one of \cite[Proposition~4.2]{CLMT21}.

\begin{prop}
  Let $\eta, \tau, \lambda \in \mathbb{R}$ and $d$ be given by~\eqref{eq:d_def}. Then, the following holds:
  \begin{itemize}
    \item[(i)] If $d \neq 0$, then $z \in \sigma_\textup{p}(A_{\eta, \tau, \lambda}^{m, c})$ if and only if $z \in \sigma_\textup{p}(A_{-4 \eta/d, -4\tau/d, -4 \lambda/d}^{m, c})$.
    \item[(ii)] $z \in \sigma_\textup{p}(A_{\eta, \tau, \lambda}^{m, c})$ if and only if $-z \in \sigma_\textup{p}(A_{-\eta, \tau, -\lambda}^{m, c})$.
  \end{itemize}
\end{prop}

\section{Proof of Theorems~\ref{theorem_nonrelativistic_limit} and~\ref{theorem_frymark_holzmann_lotoreichik}} \label{section_nonrelativistic_limit}

Throughout this section we assume that $m > 0$. First,  we follow ideas from \cite[Section~5]{BEHL18} and show Theorem~\ref{theorem_nonrelativistic_limit} about the nonrelativistic limit of $A_{\eta/2, \eta/2, 0}^{m, c}$ and the geometrically induced bound states for large $c$, for which some preparations are necessary.
Recall that $H_\eta$, $\eta \in \mathbb{R}$, is the self-adjoint operator associated with the quadratic form $\mathfrak{h}_\eta$ defined by~\eqref{def_Schroedinger_delta}.
To state the resolvent formula for $H_\eta$, define for $z \in \mathbb{C} \setminus [0, +\infty)$ the single layer potential $SL(z): L^2(\Gamma; \mathbb{C}) \rightarrow L^2(\mathbb{R}^2; \mathbb{C})$ acting as
\begin{equation} \label{def_single_layer_potential}
  SL(z) \varphi(x) = \frac{m}{\pi} \int_\Gamma K_0 \left(- i \sqrt{2 m z} | x - y |\right) \varphi(y) d \sigma(y), \quad x \in \mathbb{R}^2 \setminus \Gamma,~ \varphi \in L^2(\Gamma; \mathbb{C}),
\end{equation}
and the single layer boundary integral operator $\mathcal{S}(z): L^2(\Gamma; \mathbb{C}) \rightarrow L^2(\Gamma; \mathbb{C})$ given by
\begin{equation} \label{def_single_layer_bio}
  \mathcal{S}(z) \varphi(x) = \frac{m}{\pi} \int_\Gamma K_0 \left(- i \sqrt{2 m z} | x - y |\right) \varphi(y) d \sigma(y), \quad x \in \Gamma,~ \varphi \in L^2(\Gamma; \mathbb{C}).
\end{equation}
It is well-known that both operators $SL(z)$ and $\mathcal{S}(z)$ are well-defined and bounded, see, e.g., \cite[Theorem~2.8~(a)]{BEHL17} applied for $d = 2$ for a reference under the present assumptions. In the following, we denote by $-\Delta$ the free Laplacian defined on $H^2(\mathbb{R}^2; \mathbb{C})$. Then, one has the following relation of $SL(z)$ and $\mathcal{S}(z)$ and the resolvent of $H_\eta$, see \cite[Lemma~2.3]{BEKS94} or \cite[Theorem~2.8]{BEHL17}:

\begin{lem} \label{lemma_Krein_Schroedinger}
  Let $\eta \in \mathbb{R}$ and $H_\eta$ be the self-adjoint operator associated with the quadratic form $\mathfrak{h}_\eta$ in~\eqref{def_Schroedinger_delta}. If  $z  \in \mathbb{C} \setminus [0, \infty)$ and $0 \in \rho(I + \eta \mathcal{S}(z))$, then $z \in \rho(H_\eta)$ and
  \begin{equation*}
    (H_\eta - z)^{-1} = \left(-\frac{1}{2 m} \Delta - z \right)^{-1} - SL(z) \big(I + \eta \mathcal{S}(z)\big)^{-1} \eta SL(\overline{z})^*.
  \end{equation*}
  In particular, there exists $z_0 < 0$ such that for any $z \in (-\infty, z_0)$ one has $0 \in \rho(I + \eta \mathcal{S}(z))$.
\end{lem}

In the following proposition we state some preliminary estimates for the convergence of the operators $\Phi_{z + m c^2}^{m, c}$ and $\mathcal{C}_{z + mc^2}^{m, c}$ given by~\eqref{def_Phi_z} and~\eqref{def_C_z_formal}, as $c \rightarrow \infty$, which are necessary to compute the desired nonrelativistic limit. In the following we make use of
\begin{equation} \label{def_e}
  P_+ := \begin{pmatrix} 1 & 0 \\ 0 & 0 \end{pmatrix} \in \mathbb{C}^{2 \times 2} \quad \text{and} \quad e := \begin{pmatrix} 1 \\ 0 \end{pmatrix} \in \mathbb{C}^2.
\end{equation}
Note that $P_+ = e e^\top$.

\begin{prop} \label{proposition_convergence_Phi_C}
  Let $z \in \mathbb{C} \setminus [0, \infty)$ and $c > \sqrt{\frac{|z|}{|m|}}$. Then, there exists a constant $K$ depending on $m, z, \Gamma$ such that
  \begin{subequations}
    \begin{align}
      \label{estimate_A_0} 
      \left\| \left(A_0^{m, c} - (z +  m c^2) \right)^{-1} - \left( -\frac{1}{2 m} \Delta - z \right)^{-1} P_+ \right\|_{L^2(\mathbb{R}^2; \mathbb{C}^2) \rightarrow L^2(\mathbb{R}^2; \mathbb{C}^2)} &\leq \frac{K}{c}, \\
      \label{estimate_Phi_z}
      \big\| \Phi_{z + m c^2}^{m, c} e - SL(z) e \big\|_{L^2(\Gamma; \mathbb{C}) \rightarrow L^2(\mathbb{R}^2; \mathbb{C}^2)} &\leq \frac{K}{c}, \\
      \label{estimate_Phi_z_star}
      \big\| e^\top (\Phi_{z + m c^2}^{m, c})^* - e^\top SL(z)^* \big\|_{L^2(\mathbb{R}^2; \mathbb{C}^2) \rightarrow L^2(\Gamma; \mathbb{C})} &\leq \frac{K}{c}, \\
      \label{estimate_C_z}
      \big\| e^\top \mathcal{C}_{z + m c^2}^{m, c} e - \mathcal{S}(z) \big\|_{L^2(\Gamma; \mathbb{C}) \rightarrow L^2(\Gamma; \mathbb{C})} &\leq \frac{K}{c}.
    \end{align}
  \end{subequations}
\end{prop}
\begin{proof}
  First, the estimate in~\eqref{estimate_A_0} is well-known, see, e.g., \cite[Lemma~3.1]{BHS23_1} for a proof in the case $m = \frac{1}{2}$. To prove~\eqref{estimate_Phi_z}, we note first that $\Phi_{z + m c^2}^{m, c} e - SL(z) e$ acts on $\varphi \in L^2(\Gamma; \mathbb{C})$ as
  \begin{equation*}
    (\Phi_{z + m c^2}^{m, c} e - SL(z) e) \varphi(x)  = \int_\Gamma \big(k_1(x - y) + k_2(x - y) \big) \varphi(y) d \sigma(y), \quad x \in \mathbb{R}^2,
  \end{equation*}
  with
  \begin{equation} \label{def_k_1_2}
    \begin{split}
      k_1(x) &:= \frac{1}{2 \pi} \left(\left( 2 m + \frac{z}{c^2} \right) K_0 \left(- i \sqrt{\tfrac{z^2}{c^2} + 2 m z} | x |\right) - 2 m K_0 \left(- i \sqrt{2 m z} | x |\right) \right) e, \\
      k_2(x) &:= \frac{1}{2 \pi c} \sqrt{\tfrac{z^2}{c^2} + 2 m z} K_1 \left( - i \sqrt{\tfrac{z^2}{c^2} + 2 m z} | x |\right)  \frac{x_1 + i x_2}{ | x | } \begin{pmatrix} 0 \\ 1 \end{pmatrix}.
    \end{split}
  \end{equation}
  In the same way as in the proof of \cite[Lemma~3.1]{BHS23_1} (see the considerations on the functions $t_1, t_2, t_3$ there) one finds that there exists constant $M_1, \kappa_1, R > 0$ such that
  \begin{equation} \label{estimate_k_1_2}
    \big|k_1(x) + k_2(x)\big|^2 \leq \frac{M_1^2}{c^2} k_3(x) k_4(x),
  \end{equation}
  where 
  \begin{equation*}
    k_3(x) := \begin{cases} |x|^{-3/2}, & \text{if } |x| \leq R, \\ e^{- \kappa_1 |x|}, & \text{if } |x| > R, \end{cases}
  \end{equation*}
  and
  \begin{equation*}
    k_4(x) := \begin{cases} |x|^{-1/2}, & \text{if } |x| \leq R, \\ e^{- \kappa_1 |x|}, & \text{if } |x| > R. \end{cases}
  \end{equation*}
  Clearly, by the translation invariance of the Lebesgue measure we have for any $y \in \Gamma$ that
  \begin{equation*} 
    \int_{\mathbb{R}^2} |k_3(x-y)| d x =  \int_{\mathbb{R}^2} |k_3(x)| d x 
  \end{equation*}
  and thus,
  \begin{equation} \label{estimate_k_3}
    M_2 := \sup_{y \in \Gamma} \int_{\mathbb{R}^2} |k_3(x-y)| d x = \int_{\mathbb{R}^2} |k_3(x)| d x < \infty.
  \end{equation}
  Next, recall that $\gamma$ is the arc-length parametrization of $\Gamma$ defined in~\eqref{arc_length_Gamma_intro} and choose for a fixed $x \in \mathbb{R}^2$ a number $t = t(x) \in \mathbb{R}$ such that $|x - \gamma(t)| = \min_{s \in \mathbb{R}} |x - \gamma(s)|$. Then, taking~\eqref{equation_bi_Lipschitz} into account, we see that 
  \begin{equation*}
    \frac{C_1}{2} |t(x) - s| \leq \frac{1}{2} |\gamma(t(x)) - \gamma(s)| \leq |x - \gamma(s)|.
  \end{equation*}
  This implies $B_R(x) \cap \Gamma \subset \{ \gamma(s): |t(x) - s| \leq \frac{2 R}{C_1} \}$, where $B_R(x)$ is the ball of radius $R$ around $x$,  and 
  \begin{equation} \label{estimate_k_4}
    \begin{split}
      M_3 &:= \sup_{x \in \mathbb{R}^2} \int_\Gamma |k_4(x-y)| d \sigma(y) = \sup_{x \in \mathbb{R}^2} \int_\mathbb{R} |k_4(x-\gamma(s))| d s \\
      &= \sup_{x \in \mathbb{R}^2} \left( \int_{\gamma^{-1}(B_R(x) \cap \Gamma)} |x-\gamma(s)|^{-1/2} d s + \int_{\gamma^{-1}( \Gamma \setminus B_R(x))} e^{-\kappa_1| x-\gamma(s)|} d s \right) \\
      &\leq \sup_{x \in \mathbb{R}^2} \left( \int_{t(x)-2R/C_1}^{t(x) + 2 R/C_1} |x-\gamma(s)|^{-1/2} d s + \int_\mathbb{R} e^{-\kappa_1| x-\gamma(s)|} d s \right) \\
      &\leq \sup_{x \in \mathbb{R}^2} M_4 \left( \int_{t(x)-2R/C_1}^{t(x) + 2 R/C_1} |t(x)-s|^{-1/2} d s + \int_{\mathbb{R}} e^{-\kappa_2 | t(x)-s|} d s \right) < +\infty,
    \end{split}
  \end{equation}
  where $M_4, \kappa_2>0$ are suitable constants.
  Using~\eqref{estimate_k_1_2}--\eqref{estimate_k_4} in the Schur test, see, e.g., \cite[Theorem 6.24]{Weidmann} or \cite[Example III 2.4]{K95}, we get
  \begin{equation*}
     \big\| \Phi_{z + m c^2}^{m, c} e - SL(z) e \big\|_{L^2(\Gamma; \mathbb{C}) \rightarrow L^2(\mathbb{R}^2; \mathbb{C}^2)} \leq \frac{M_1}{c} \sqrt{M_2 M_3},
  \end{equation*}
  which is exactly~\eqref{estimate_Phi_z}. By taking adjoints this also implies~\eqref{estimate_Phi_z_star}.
  
  It remains to prove~\eqref{estimate_C_z}. Note that for $\varphi \in L^2(\Gamma; \mathbb{C})$ one has
  \begin{equation*}
    \big(e^\top \mathcal{C}_{z + m c^2}^{m, c} e - \mathcal{S}(z) \big) \varphi(x) = \int_\Gamma e^\top k_1(x-y) \varphi(y) d \sigma(y), \quad x \in \Gamma,
  \end{equation*}
  where $k_1$ is given by~\eqref{def_k_1_2}. The function $e^\top k_1$ can be further decomposed in  $e^\top k_1 = k_5 + k_6$ with 
  \begin{equation*}
    \begin{split}
      k_5(x) &:= \frac{1}{2 \pi} \left( 2 m + \frac{z}{c^2} \right) \left( K_0 \left(- i \sqrt{\tfrac{z^2}{c^2} + 2 m z} | x |\right) -  K_0 \left(- i \sqrt{2 m z} | x |\right) \right), \\
      k_6(x) &:= \frac{z}{c^2} K_0 \left(- i \sqrt{2 m z} | x |\right).
    \end{split}
  \end{equation*}
  With the same argument as in \cite[equation~(3.11)]{BHS23_1} one sees that 
  \begin{equation*}
    |k_5(x)| \leq \frac{M_5}{c} \begin{cases} 1, &\text{if } |x| \leq R, \\ e^{- \kappa_3 |x|}, & \text{if } |x| > R, \end{cases}
  \end{equation*}
  while for $k_6$ one gets with the asymptotics for the Bessel function $K_0$ in \cite[Sections~9.6 and~9.7]{AS84} that
  \begin{equation*}
    |k_6(x)| \leq \frac{M_5}{c} \begin{cases} |\log|x|| + 1, &\text{if } |x| \leq R, \\ e^{- \kappa_3 |x|}, & \text{if } |x| > R, \end{cases}
  \end{equation*}
  where $M_5, \kappa_3 > 0$ are suitable constants.
  Note that for a fixed $y = \gamma(s) \in \Gamma$, $s \in \mathbb{R}$, we have $B_R(\gamma(s)) \cap \Gamma \subset \{ \gamma(t): |t - s| \leq \frac{R}{C_1} \}$ by \eqref{equation_bi_Lipschitz}. Hence, one obtains as in~\eqref{estimate_k_4} with constants $M_6, M_7, M_8, \kappa_4 > 0$ that
  \begin{equation*} 
    \begin{split}
      &\sup_{x \in \Gamma} \int_\Gamma |k_1(x-y)| d \sigma(y) = \sup_{y \in \Gamma} \int_\Gamma |k_1(x-y)| d \sigma(x) = \sup_{y \in \Gamma} \int_\mathbb{R} |k_1(\gamma(s) - y)| d s \\
      ~&\leq \sup_{y \in \Gamma} \frac{M_5}{c} \left( \int_{\gamma^{-1}(B_R(x) \cap \Gamma)} \big(\big|\log |\gamma(s) - y| \big| + 1\big) d s + \int_{\gamma^{-1}( \Gamma \setminus B_R(x))} e^{-\kappa_3| \gamma(s) - y|} d s \right) \\
      ~&\leq \sup_{y = \gamma(t) \in \Gamma} \frac{M_6}{c} \left( \int_{t - R/C_1}^{t + R/C_1} \big(|\gamma(s) - \gamma(t)|^{-1/2} + 1\big) d s + \int_\mathbb{R} e^{-\kappa_3| \gamma(s) - \gamma(t)|} d s \right) \\
      ~&\leq \sup_{y = \gamma(t) \in \Gamma} \frac{M_7}{c} \left( \int_{t - R/C_1}^{t + R/C_1} \big(|s - t|^{-1/2} + 1\big) d s + \int_\mathbb{R} e^{-\kappa_4| s - t|} d s \right) \\
      ~&\leq \frac{M_8}{c} < +\infty.
    \end{split}
  \end{equation*}
  This yields, together with the Schur test, the claim in~\eqref{estimate_C_z}.
\end{proof}

Using the result from Proposition~\ref{proposition_convergence_Phi_C} it is not difficult to prove that the nonrelativistic limit of $A_{\eta/2, \eta/2, 0}^{m, c}$ is $H_\eta$. Recall that the matrix $P_+$ is defined in~\eqref{def_e}.

\begin{proof}[Proof of Theorem~\ref{theorem_nonrelativistic_limit}]
  First, we prove the result about the nonrelativistic limit.
  Let $z_0$ be as in Lemma~\ref{lemma_Krein_Schroedinger}, $z < z_0$, and $c$ be sufficiently large so that $z + mc^2 \in \rho(A_0^{m, c})$.
  Note that~\eqref{estimate_C_z}  implies that 
  \begin{equation*}
    \big\| (I + \eta \mathcal{S}(z)) - (I + \eta e^\top \mathcal{C}_{z + m c^2}^{m, c} e) \big\|_{L^2(\Gamma; \mathbb{C}) \rightarrow L^2(\Gamma; \mathbb{C})} \leq \frac{K}{c}
  \end{equation*}
  and hence, for sufficiently large $c > 0$ the operator $I + \eta e^\top \mathcal{C}_{z + m c^2}^{m, c} e$ is boundedly invertible in $L^2(\Gamma; \mathbb{C})$, as $I + \eta \mathcal{S}(z)$ has this property by Lemma~\ref{lemma_Krein_Schroedinger}. Applying \cite[Theorem~IV~1.16]{K95} and~\eqref{estimate_C_z} yields
  \begin{equation} \label{convergence_inverse}
    \big\| (I + \eta \mathcal{S}(z))^{-1} - (I + \eta e^\top \mathcal{C}_{z + m c^2}^{m, c} e)^{-1} \big\|_{L^2(\Gamma; \mathbb{C}) \rightarrow L^2(\Gamma; \mathbb{C})} \leq \frac{K}{c}.
  \end{equation}
  Since $I + \eta e^\top \mathcal{C}_{z + m c^2}^{m, c} e$ is boundedly invertible, also
  \begin{equation*}
    \sigma_0 + \eta e e^\top \mathcal{C}_{z + m c^2}^{m, c} = \sigma_0 + \left(\frac{\eta}{2} \sigma_0 + \frac{\eta}{2} \sigma_3\right) \mathcal{C}_{z + m c^2}^{m, c}
  \end{equation*}
  is boundedly invertible in $L^2(\Gamma; \mathbb{C}^2)$. Thus, it follows from Proposition~\ref{proposition_Krein_formula} that $z \in \rho(A_{\eta/2, \eta/2, 0}^{m, c})$ and
  \begin{equation*} 
    \begin{split}
      \big(A_{\eta/2, \eta/2, 0}^{m, c} - (z + mc^2) \big)^{-1} &= \big(A_0^{m, c} - (z + m c^2) \big)^{-1} \\
        &\qquad - \Phi_{z + m c^2}^{m, c} \big(\sigma_0 + \eta e e^\top \mathcal{C}_{z + m c^2}^{m, c} \big)^{-1} \eta e e^\top (\Phi_{\overline{z} + m c^2}^{m, c})^* \\
        &= \big(A_0^{m, c} - (z + m c^2) \big)^{-1} \\
        &\qquad - \Phi_{z + m c^2}^{m, c} e \big(I + \eta e^\top \mathcal{C}_{z + m c^2}^{m, c} e \big)^{-1} \eta e^\top (\Phi_{\overline{z} + m c^2}^{m, c})^*.
    \end{split}
  \end{equation*}
  Using  Proposition~\ref{proposition_convergence_Phi_C} and~\eqref{convergence_inverse} in this equation, one finds 
  \begin{equation*} 
    \begin{split}
      \lim_{c \rightarrow +\infty} \big(A_{\eta/2, \eta/2, 0}^{m, c} - (z + mc^2) \big)^{-1} &= \left( -\frac{1}{2 m} \Delta - z \right)^{-1}  P_+\\
        &\qquad - SL(z) e \big(I + \eta \mathcal{S}(z) \big)^{-1}  \eta e^\top SL(\overline{z})^*
    \end{split}
  \end{equation*}
  and that the order of convergence is $\frac{1}{c}$. Since $e e^\top = P_+$, we see with Lemma~\ref{lemma_Krein_Schroedinger} that the last expression is $(H_\eta - z)^{-1}  P_+$.

  To show $\sigma_\textup{disc}(A_{\eta/2, \eta/2, 0}^{m, c})\not=\emptyset$ in the case $\eta < 0$ and $\Gamma$ is not the straight line, we follow closely the proof of  \cite[Proposition~5.5]{BEHL18} and transfer the result from \cite{EI01} about the spectrum of $H_\eta$ to $A_{\eta/2, \eta/2, 0}^{m, c}$.  
  Note first that, due to the assumption $\eta < 0$ and $\Gamma$ is not the straight line, $\sigma_{\textup{disc}}(H_\eta P_+) \neq \emptyset$, see \cite[Theorem~5.2]{EI01}. 
  Now fix some $z < z_0$. We remark that $(H_\eta - z)^{-1} P_+$ is the resolvent of a self-adjoint relation (multivalued operator) in $L^2(\mathbb{R}^2; \mathbb{C}^2)$, whose operator part is $H_\eta P_+$.  
  Since $(A_{\eta/2, \eta/2, 0}^{m, c} - (z -+ m c^2))^{-1}$ converges in the operator norm to $(H_\eta - z)^{-1} P_+$, as $c \rightarrow +\infty$, it follows in the same way as in \cite[Theorem~VIII.23~(b)]{RSI} or \cite[Satz~9.24~b)]{W00} that the spectral measures of $A_{\eta/2, \eta/2, 0}^{m,c} - mc^2$ converge to the spectral measures of $H_\eta P_+$. Thus, one finds with \cite[Satz~2.58~a)]{W00} that also $\sigma_{\textup{disc}}(A_{\eta/2, \eta/2, 0}^{m, c}) \neq \emptyset$.
\end{proof}

\begin{remark} \label{remark_nonrelativistic_limit}
  One can show, in a similar way as in Theorem~\ref{theorem_nonrelativistic_limit}, that $A_{\eta, \tau, \lambda}^{m, c}$ converges in the nonrelativistic limit to $H_{\eta + \tau}$. Hence, as long as $\Gamma$ is not the straight line and $\eta + \tau < 0$, one concludes as above that $\sigma_{\textup{disc}}(A_{\eta, \tau, \lambda}^{m, c}) \neq \emptyset$ in this situation. However, since the coefficients of the Lorentz scalar and the anomalous magnetic interactions should be scaled differently to obtain a nonrelativistic limit modelling the same physics, we prefer to state the simpler result on $A_{\eta/2, \eta/2, 0}^{m, c}$ in this section.
\end{remark}

Next, we prove Theorem~\ref{theorem_frymark_holzmann_lotoreichik}. 
Recall that $\widetilde{\Gamma}_\omega$ is the broken line defined in~\eqref{def_broken_line} and that the essential spectrum of $A_{0, \tau, 0}^{m, c}$ was identified in  Corollary~\ref{corollary_Lorentz_scalar}.

\begin{proof}[Proof of Theorem~\ref{theorem_frymark_holzmann_lotoreichik}]
  In the following we denote by $\widetilde{\Omega}_\pm$ the two disjoint  domains with joint boundary $\widetilde{\Gamma}_\omega$, which are chosen such that $(1, 0) \in \widetilde{\Omega}_+$, and we denote by $\widetilde{\nu}$ the unit normal vector field at $\widetilde{\Gamma}_\omega$ that is pointing outwards of $\widetilde{\Omega}_+$. We will make use of the operator 
  \begin{equation*}
    \begin{split}
      \widetilde{S}_\tau^{m, c} f &:= \big(-i c (\sigma \cdot \nabla) + m c^2 \sigma_3 \big) f_+ \oplus \big(-i c (\sigma \cdot \nabla) + m c^2 \sigma_3 \big) f_-, \\
      \dom \widetilde{S}_\tau^{m, c} &:= \bigg\{ f = f_+ \oplus f_- \in H^1(\widetilde{\Omega}_+; \mathbb{C}^2) \oplus H(\widetilde{\Omega}_-; \mathbb{C}^2): \\
      &\qquad \qquad \quad  -i c (\sigma \cdot \widetilde{\nu}) (f_+|_{\widetilde{\Gamma}_\omega} - f_-|_{\widetilde{\Gamma}_\omega}) = \frac{\tau}{2} \sigma_3 (f_+|_{\widetilde{\Gamma}_\omega} + f_-|_{\widetilde{\Gamma}_\omega}) \bigg\}.
    \end{split}
  \end{equation*}
  Note that $\widetilde{S}_\tau^{m, c} = c \widetilde{S}_{\tau/c}^{m c, 1} = c S_{\tau/c}$, where $S_\tau$ is the operator considered in \cite{FHL23}, when $m$ is replaced by $m c$. Hence, by \cite[Section~2]{FHL23} the operator $\widetilde{S}_\tau^{m, c}$ is closed and symmetric,
  but it has deficiency indices $\dim \ker ((\widetilde{S}_\tau^{m, c})^* \mp i) = 1$, i.e. $\widetilde{S}_\tau^{m, c}$ is symmetric, but not self-adjoint. 
  
  Next, fix $L_0 > 0$ such that $\Gamma \cap ([L_0, +\infty) \times \mathbb{R}) = \widetilde{\Gamma}_\omega \cap ([L_0, +\infty) \times \mathbb{R})$, i.e. $\Gamma$ and $\widetilde{\Gamma}_\omega$ coincide for $x \geq L_0$. Consider for $L > L_0$ and $c_1, \dots, c_N \in \mathbb{C}$ the function
  \begin{equation*}
    f(x,y) := \sum_{n=1}^N c_n u_n(x) v(y) w(x,y)
  \end{equation*}
  with
  \begin{equation*}
    u_n(x) := \sin \left( \frac{2 n \pi}{L} x \right) \chi_{[L, 2L]}(x),
  \end{equation*}
  where $\chi_{[L, 2L]}$ denotes the indicator function for $[L, 2L]$,
  \begin{equation*}
    v(y) := \begin{cases} 1, & |y| \leq 2 r, \\ e^{-\gamma(|y| - 2r)}, & |y| > 2r, \end{cases}
  \end{equation*}
  where $r := L \tan (\omega)$ and $\gamma := -\frac{4 m c^2 \tau}{4 c^2 + \tau^2} > 0$, and
  \begin{equation*}
    w(x, y) := \begin{cases} \begin{pmatrix} 1 \\ 0 \end{pmatrix}, & (x, y) \in \Omega_+, \\  \begin{pmatrix} \tfrac{4 c^2 + \tau^2}{4 c^2 - \tau^2} \\ \tfrac{4 c \tau}{4 c^2 - \tau^2} e^{i \omega} \end{pmatrix}, & (x, y) \in \Omega_- ~\wedge~ y > 0, \\ \begin{pmatrix} \tfrac{4 c^2 + \tau^2}{4 c^2 - \tau^2} \\ -\tfrac{4 c \tau}{4 c^2 - \tau^2} e^{-i \omega} \end{pmatrix}, & (x, y) \in \Omega_- ~\wedge~ y < 0. \end{cases}
  \end{equation*}
  As in the proof of \cite[Proposition~6.3]{FHL23} one sees that, when varying the coefficients $c_1, \dots, c_N$, the functions $f$ span an $N$ dimensional space and since each $f$ is supported in $[L, 2 L] \times \mathbb{R}$, we find that $f \in \dom \widetilde{S}_\tau^{m, c} \cap \dom A_{0, \tau, 0}^{m, c}$ and $\widetilde{S}_\tau^{m, c}f=A_{0, \tau, 0}^{m, c}f$.  Recall that by the scaling property $\widetilde{S}_\tau^{m, c} = c S_{\tau/c}$ one has to replace $m$ and $\tau$ 
in the formulas in \cite{FHL23} by $mc$ and $\tau/c$, respectively.
  Hence, one gets with \cite[equation~(6.3)]{FHL23}
  \begin{equation*}
    \begin{split}
      \| A_{0, \tau, 0}^{m, c} f &\|_{L^2(\mathbb{R}^2; \mathbb{C}^2)}^2 - \left( m c^2 \frac{\tau^2 - 4 c^2}{\tau^2 + 4 c^2} \right)^2 \|f \|_{L^2(\mathbb{R}^2; \mathbb{C}^2)}^2 \\
      &= \| \widetilde{S}_\tau^{m, c} f \|_{L^2(\mathbb{R}^2; \mathbb{C}^2)}^2 - \left( m c^2 \frac{\tau^2 - 4 c^2}{\tau^2 + 4 c^2} \right)^2 \|f \|_{L^2(\mathbb{R}^2; \mathbb{C}^2)}^2 \\
      &= c^2 \left( \| S_{\tau/c} f \|_{L^2(\mathbb{R}^2; \mathbb{C}^2)}^2 - \left( m c \frac{\tau^2 - 4 c^2}{\tau^2 + 4 c^2} \right)^2 \|f \|_{L^2(\mathbb{R}^2; \mathbb{C}^2)}^2 \right) \\
      &< c^2 \sum_{n=1}^N |c_n|^2 \bigg\{ \tan(\omega) \left[ 3 + \frac{(4 c^2 + \tau^2)^2 + 16 c^2 \tau^2}{(4 c^2 - \tau^2)^2} \right] (2 N^2 \pi^2 + m^2 c^2 L^2) \\
      &\qquad \qquad \qquad \qquad  + \frac{4 m c^2 L \tau}{4 c^2 + \tau^2} - \frac{N^2 \pi^2 [(4 c^2 + \tau^2)^2 + 16 \tau^2 c^2] (4 c^2 + \tau^2)}{2 m c^2 L \tau (4c^2 - \tau^2)^2} \bigg\}.
    \end{split}
  \end{equation*}
Since we consider $m>0$ and $\tau<0$, the sum of the second and the third term in the curly brackets is negative, when $L$ is sufficiently large. For these values of $L$, one can choose $\omega_\star = \omega_\star(L)$ such that the right hand side of the last displayed formula is negative for any $\omega < \omega_\star$. As the space of all $f$'s has dimension $N$, we find with the min-max principle that $(A_{0, \tau, 0}^{m, c})^2$ has at least $N$ discrete eigenvalues below the bottom of its essential spectrum $\big( m c^2 \frac{\tau^2 - 4 c^2}{\tau^2 + 4 c^2} \big)^2$; cf. Corollary~\ref{corollary_Lorentz_scalar}. Therefore, by the spectral theorem, $A_{0, \tau, 0}^{m, c}$ has at least $N$ discrete eigenvalues in the gap of its essential spectrum. This finishes the proof of this proposition.
\end{proof}

\subsection*{Acknowledgements} 
J. Behrndt and M. Holzmann gratefully acknowledge financial support by the Austrian Science Fund (FWF): P 33568-N. P.~Exner and M.~Tu\v{s}ek were supported by the grant No.~21-07129S of the Czech Science Foundation (GA\v{C}R). J.~Behrndt, M.~Holzmann, and M.~Tu\v{s}ek thank the Ministry of Education, Youth and Sports of the Czech Republic and the Austrian Agency for International Cooperation in Education and Research (OeAD) for support within the project 8J23AT025 and CZ 16/2023, respectively. This publication is based upon work from COST Action CA 18232 MAT-DYN-NET, supported by COST (European Cooperation in Science and Technology), www.cost.eu.

\begin{appendix}
\section{Boundary triples} \label{appendix_bt}

In this section we briefly recall some notions related to the theory of (ordinary) boundary triples. Boundary triples are an abstract tool in the extension and spectral theory of symmetric and self-adjoint operators and they are particularly useful to handle boundary value and transmission problems. Here, we only state the definitions and results needed in the main part of the paper, for proofs and further results we refer to \cite{BHS19, BGP, DM91, DM95}.

Throughout this section $\mathcal{H}$ is a complex Hilbert space and $S$ is a densely defined closed symmetric operator in $\mathcal{H}$.

\begin{definition} \label{definition_bt}
  Let $\mathcal{G}$ be a complex Hilbert space and $\Gamma_0, \Gamma_1: \dom S^* \rightarrow \mathcal{G}$ be linear mappings. Then $\{ \mathcal{G}, \Gamma_0, \Gamma_1 \}$ is called a boundary triple for $S^*$, if the following conditions are fulfilled:
  \begin{itemize}
    \item[(i)] For all $f, g \in \dom S^*$ the abstract Green's identity
    \begin{equation*}
      (S^* f, g)_\mathcal{H} - (f, S^* g)_\mathcal{H} = (\Gamma_1 f, \Gamma_0 g)_\mathcal{G} - (\Gamma_0 f, \Gamma_1 g)_\mathcal{G}
    \end{equation*}
    holds.
    \item[(ii)] The map $(\Gamma_0, \Gamma_1): \dom S^* \rightarrow \mathcal{G} \times \mathcal{G}$ is surjective.
  \end{itemize}
\end{definition}

We note that a boundary triple for $S^*$ exists if and only if $S$ admits self-adjoint extensions. In the following we always  assume that this is the case and that $\{ \mathcal{G}, \Gamma_0, \Gamma_1 \}$ is a boundary triple associated with $S^*$. Then, the extension $B_0 := S^* \upharpoonright \ker \Gamma_0$ of $S$ is self-adjoint and one has 
the direct sum decomposition
\begin{equation*}
  \dom S^* = \dom B_0 \dot{+} \ker (S^* - z) = \ker \Gamma_0 \dot{+} \ker(S^*-z), \quad z \in \rho(B_0).
\end{equation*}
Hence, the following mappings are well defined: The $\gamma$-field
\begin{equation*}
  \rho(B_0) \ni z \mapsto \gamma(z) := \big(\Gamma_0 \upharpoonright \ker (S^*-z) \big)^{-1},
\end{equation*}
whose values $\gamma(z)$, $z \in \rho(B_0)$, are bounded and everywhere defined operators from $\mathcal{G}$ to $\mathcal{H}$, and the Weyl function
\begin{equation*}
  \rho(B_0) \ni z \mapsto M(z) := \Gamma_1 \big(\Gamma_0 \upharpoonright \ker (S^*-z) \big)^{-1},
\end{equation*}
whose values $M(z)$, $z \in \rho(B_0)$, are bounded and everywhere defined operators in $\mathcal{G}$. We remark that 
\begin{equation} \label{gamma_star}
  \gamma(z)^* = \Gamma_1 (B_0 - \overline{z})^{-1},\quad z \in \rho(B_0).
\end{equation}

In the following we discuss how the self-adjointness of certain extensions of $S$ can be shown and how their spectral properties can be analyzed with the help of a boundary triple. Let $\Theta$ be a linear operator in $\mathcal{G}$. Then we define in $\mathcal{H}$ the linear operator $B_{\Theta}$ as the restriction of $S^*$ onto
\begin{equation} \label{def_B_Theta}
  \dom B_{\Theta} := \big\{ f \in \dom S^*: \Gamma_1 f = \Theta \Gamma_0 f \big\}.
\end{equation}
Note that in the above definition the equation $\Gamma_1 f = \Theta \Gamma_0 f$ contains the condition $\Gamma_0 f \in \dom \Theta$. 
We remark that for the description of all closed extensions of $S$ via \eqref{def_B_Theta} it is necessary to extend the class of parameters $\Theta$ to the class of closed linear relations in $\mathcal G$; cf. \cite[Chapter 2.2]{BHS19}.
The following theorem shows that several properties of $B_{\Theta}$ are encoded in the parameter $\Theta$.

\begin{thm} \label{theorem_bt}
  Let $\{ \mathcal{G}, \Gamma_0, \Gamma_1 \}$ be a boundary triple for $S^*$ with associated $\gamma$-field $\gamma$ and Weyl function $M$. Let $\Theta$ be a linear operator in $\mathcal{G}$ and let $B_\Theta$ be defined by~\eqref{def_B_Theta}. Then $B_{\Theta}$ is self-adjoint in $\mathcal{H}$ if and only if $\Theta$ is self-adjoint in $\mathcal{G}$. If $\Theta$ is self-adjoint, then the following holds for all $z \in \rho(B_0)$:
  \begin{itemize}
    \item[(i)] $z \in \sigma(B_{\Theta})$ if and only if $0 \in \sigma(\Theta - M(z) )$.
    \item[(ii)] $z \in \sigma_\textup{p}(B_{\Theta})$ if and only if $0 \in \sigma_\textup{p}(\Theta - M(z) )$.
    \item[(iii)] $z \in \sigma_\textup{ess}(B_{\Theta})$ if and only if $0 \in \sigma_\textup{ess}(\Theta - M(z) )$.
    \item[(iv)] For $z \in \rho(B_{\Theta})$ one has
    \begin{equation*}
      (B_{\Theta} - z)^{-1} = (B_0 - z)^{-1} + \gamma(z) \big( \Theta - M(z) \big)^{-1} \gamma(\overline{z})^*.
    \end{equation*}
  \end{itemize}
\end{thm}

Finally, we shortly recall an efficient way how a boundary triple can be constructed, see \cite{P04, P08} or also \cite[Section~1.4.2]{BGP}. Let $B$ be a self-adjoint operator in a Hilbert space $\mathcal{H}$, $\mathcal{G}$ be another Hilbert space, and assume that $\mathcal{T}: \dom B \rightarrow \mathcal{G}$ is surjective, continuous (when $\dom B$ is endowed with the graph norm of $B$), and that $\ker \mathcal{T}$ is dense in $\mathcal{H}$. Then, $S := B \upharpoonright \ker \mathcal{T}$ is a densely defined closed symmetric operator in $\mathcal{H}$. 

To proceed, define for $z \in \rho(B)$ the operator
\begin{equation} \label{def_G_z_bt}
  \mathcal{G}_z := \big( \mathcal{T} (B - \overline{z})^{-1} \big)^*.
\end{equation}
It can be shown that $\mathcal{G}_z$ is bounded and injective. Moreover, one has $\ran \mathcal{G}_z  = \ker (S^* - z)$ and that for any $z \in \rho(B)$ the direct sum decomposition 
\begin{equation*}
  \dom S^* = \dom B \dot{+} \ker (S^* - z) = \dom B \dot{+} \ran \mathcal{G}_z
\end{equation*}
holds. Hence, for any $f \in \dom S^*$ there exist $f_z \in \dom B$ and $\xi \in \mathcal{G}$ such that $f = f_z + \mathcal{G}_z \xi$; it can be shown that $\xi$ is independent of the choice of $z \in \rho(B)$. Then, the following result holds:

\begin{prop} \label{proposition_boundary_triple_singular_perturbation}
  Let $\zeta \in \rho(B)$ be fixed and define the mappings $\Gamma_0, \Gamma_1: \dom S^* \rightarrow \mathcal{G}$ acting on $f = f_\zeta + \mathcal{G}_\zeta \xi = f_{\overline{\zeta}} + \mathcal{G}_{\overline{\zeta}} \xi$ with $f_\zeta, f_{\overline{\zeta}} \in \dom B$, $\xi \in \mathcal{G}$, as
  \begin{equation*}
    \Gamma_0 f = \xi \quad \text{and} \quad \Gamma_1 f = \frac{1}{2} \mathcal{T} (f_\zeta + f_{\overline{\zeta}}).
  \end{equation*}
  Then, $\{ \mathcal{G}, \Gamma_0, \Gamma_1 \}$ is a boundary triple for $S^*$. Moreover, the values of the $\gamma$-field and Weyl function are given by $\gamma(z) = \mathcal{G}_z$ and $M(z) = \mathcal{T}(\mathcal{G}_z - \frac{1}{2} (\mathcal{G}_\zeta + \mathcal{G}_{\overline{\zeta}}))$, $z \in \rho(B)$.
\end{prop}

\section{Proof of Proposition~\ref{proposition_C_z_Gamma}} \label{appendix_int_op}

First, we provide an abstract result about the compactness of integral operators under suitable assumptions on the integral kernel that will be used several times in the proof of Proposition~\ref{proposition_C_z_Gamma}.

\begin{lem} \label{lemma_int_op_compact}
  Let $k: \mathbb{R}^2 \rightarrow \mathbb{C}$ be continuous in $\mathbb{R}^2$ and differentiable on the set $\{ (s, t) \in \mathbb{R}^2: s \neq t \}$. Assume that there exist constants $R, C, \kappa > 0$ and $\alpha \in [0, 1)$ such that for $p \in \{ 0, 1\}$ and all $(s, t) \in \mathbb{R}^2$ with $s \neq t$
  \begin{equation} \label{assumption_k}
    \max \left\{ \left|\frac{\partial^p}{\partial s^p} k(s, t)\right|, \left|\frac{\partial^p}{\partial s^p} \overline{k(t, s)}\right| \right\} \leq C \begin{cases} |s - t|^{-\alpha}, & |s|, |t| \leq 3 R, \\ e^{-\kappa |s - t|}, & |s - t| > 2 R, \\ 0, & s, t > R \text{ or } s, t < -R, \end{cases}
  \end{equation}
  holds. Then, for any $r \in [-1, 0]$ the mapping 
  \begin{equation*}
    \mathcal{K} u(s) = \int_\mathbb{R} k(s,t) u(t) dt, \quad u \in C_0^\infty(\mathbb{R}; \mathbb{C}), ~ s \in \mathbb{R},
  \end{equation*}
  can be extended to a compact operator $\mathcal{K}: H^{r}(\mathbb{R}; \mathbb{C}) \rightarrow H^{r+1}(\mathbb{R}; \mathbb{C})$.
\end{lem}
\begin{proof}
  Define for $p \in \{ 0, 1 \}$ the function $k_p(s, t) := \frac{\partial^p}{\partial s^p} k(s,t)$, $s \neq t$, and the integral operator $\mathcal{K}_p$
  \begin{equation*}
    \mathcal{K}_p: L^2(\mathbb{R}; \mathbb{C}) \rightarrow L^2(\mathbb{R}; \mathbb{C}), \qquad \mathcal{K}_p u(s) := \int_\mathbb{R} k_p(s, t) u(t) d t.
  \end{equation*}
  The proof is split into several steps: In \textit{Step~1} we show that $k_p \in L^2(\mathbb{R}^2; \mathbb{C})$, which implies that $\mathcal{K}_p$ is a Hilbert-Schmidt operator in $L^2(\mathbb{R}; \mathbb{C})$ and hence, it is well-defined and compact. In \textit{Step~2} we verify that for any $u \in L^2(\mathbb{R}; \mathbb{C})$ the function $\mathcal{K} u = \mathcal{K}_0 u$ is weakly differentiable and that $\frac{d}{d s} \mathcal{K}_0 u = \mathcal{K}_1 u$. Using this, it is proved in \textit{Step~3} that $\mathcal{K}$ gives rise to a compact operator from $L^2(\mathbb{R}; \mathbb{C})$ to $H^1(\mathbb{R}; \mathbb{C})$. Then, in \textit{Step~4} it is argued that $\mathcal{K}$ admits a compact extension from $H^{-1}(\mathbb{R}; \mathbb{C})$ to $L^2(\mathbb{R}; \mathbb{C})$. Finally, in \textit{Step~5} an interpolation argument is used to deduce that also for any $r \in (-1, 0)$ the map $\mathcal{K}$ has a compact extension from $H^r(\mathbb{R}; \mathbb{C})$ to $H^{r+1}(\mathbb{R}; \mathbb{C})$.

  \textit{Step~1}.
  We claim that $k_{p} \in L^2(\mathbb{R}^2; \mathbb{C})$. For two intervals $I_1$ and $I_2$ set 
  \begin{equation*}
    \mathcal{I}(I_1, I_2) := \int_{I_1 \times I_2} |k_{p}(s,t)|^2 ds dt.
  \end{equation*}
  Then 
  \begin{equation} \label{decomposition_I_2}
    \| k_p \|_{L^2(\mathbb{R}^2; \mathbb{C})}^2 = \mathcal{I}((-\infty, -R], \mathbb{R}) + \mathcal{I}((-R, R), \mathbb{R}) + \mathcal{I}([R, +\infty), \mathbb{R}).
  \end{equation}
  We further decompose 
  \begin{equation} \label{decomposition_I_1}
    \begin{split}
      \mathcal{I}((-\infty, -R], \mathbb{R}) &= \mathcal{I}((-\infty, -R], (-\infty, -R]) + \mathcal{I}((-\infty, -3 R], (-R, R)) \\
      &\qquad+ \mathcal{I}((-3 R, -R], (-R, R)) + \mathcal{I}((-\infty, -R], [R, +\infty)).
    \end{split}
  \end{equation}
  By~\eqref{assumption_k} we have
  \begin{equation} \label{I_j_bounded1}
    \mathcal{I}((-\infty, -R], (-\infty, -R]) = 0
  \end{equation}
  and also
  \begin{equation} \label{I_j_bounded2}
     \mathcal{I}((-\infty, -R], [R, +\infty)) \leq C^2 \int_{-\infty}^{-R} \int_R^{+\infty} e^{2 \kappa (t-s)} ds dt < +\infty.
  \end{equation}
  In the same way one finds that $\mathcal{I}((-\infty, -3 R], (-R, R)) < +\infty$.
  Eventually, taking the properties of $k$ from~\eqref{assumption_k} for $|s|, |t| \leq 3 R$ into account, 
  one gets for bounded intervals $I_1, I_2 \subset [-3R, 3R]$ that
  \begin{equation} \label{I_j_bounded3}
    \mathcal{I}(I_1, I_2) < +\infty.
  \end{equation}
  Using~\eqref{I_j_bounded1}--\eqref{I_j_bounded3} in~\eqref{decomposition_I_1} one concludes that $\mathcal{I}((-\infty, -R], \mathbb{R}) < +\infty$. In exactly the same way one gets that $\mathcal{I}([R, +\infty), \mathbb{R}) < +\infty$. Finally, employing again similar estimates as in \eqref{I_j_bounded2}--\eqref{I_j_bounded3} one finds that
  \begin{equation*} 
    \begin{split}
      \mathcal{I}((-R, R), \mathbb{R}) = \mathcal{I}((-R, R), (-\infty, -3 R]) &+ \mathcal{I}((-R, R), (-3 R, 3 R)) \\
      &+ \mathcal{I}((-R, R), [3 R, +\infty)) < +\infty.
    \end{split}
  \end{equation*}
  Hence, we get from~\eqref{decomposition_I_2} that $k_{p} \in L^2(\mathbb{R}^2; \mathbb{C})$ showing that $\mathcal{K}_{p}$ is a Hilbert-Schmidt operator and thus, compact in $L^2(\mathbb{R}; \mathbb{C})$.

  \textit{Step~2}.
  Let $u \in C_0^\infty(\mathbb{R}; \mathbb{C})$ be fixed. We claim that $\mathcal{K}_0 u \in H^1(\mathbb{R}; \mathbb{C})$ and $\frac{d}{d s} \mathcal{K}_0 u = \mathcal{K}_1 u$. Taking the result from \textit{Step~1} into account, it suffices to prove the latter identity for the weak derivative. Let $\varphi \in C_0^\infty(\mathbb{R}; \mathbb{C})$. By~\eqref{assumption_k} one has for any fixed $t \in \mathbb{R}$ that $\varphi' k(\cdot, t) \in L^1(\mathbb{R}; \mathbb{C})$. 
  Hence, we can apply Fubini's theorem and the dominated convergence theorem and get
  \begin{equation*}
    \begin{split}
      \int_\mathbb{R} \varphi'(s) \mathcal{K}_0 u(s) d s &= \int_\mathbb{R} \int_\mathbb{R} \varphi'(s) k(s, t) d s \, u(t) d t \\
      &= \int_\mathbb{R} \lim_{\varepsilon \searrow 0} \int_{\mathbb{R} \setminus (t - \varepsilon, t + \varepsilon)} \varphi'(s) k(s, t) d s \, u(t) d t.
    \end{split}
  \end{equation*}
  Next, we use integration by parts, which is allowed, as for a fixed $t$ the map $s \mapsto k(s, t)$ is differentiable in $\mathbb{R} \setminus (t - \varepsilon, t + \varepsilon)$. Using the continuity of $k$ and $\varphi$, the dominated convergence theorem, which is applicable, since for any fixed $t \in \mathbb{R}$ one has $\varphi \frac{\partial}{\partial s} k(\cdot, t) \in L^1(\mathbb{R}; \mathbb{C})$ due to~\eqref{assumption_k}, and Fubini's theorem we find that
  \begin{equation*}
    \begin{split}
      \int_\mathbb{R} \varphi'(s) \mathcal{K}_0 u(s) d s &= \int_\mathbb{R} \lim_{\varepsilon \searrow 0} \left[ - \varphi(s) k(s, t)\Big|_{t - \varepsilon}^{t + \varepsilon} - \int_{\mathbb{R} \setminus (t - \varepsilon, t + \varepsilon)} \varphi(s) \frac{\partial}{\partial s} k(s, t) d s\right] u(t) d t \\
      &= -\int_\mathbb{R} \int_{\mathbb{R}} \varphi(s) \frac{\partial}{\partial s} k(s, t) d s \, u(t) d t = -\int_\mathbb{R} \varphi(s) \mathcal{K}_1 u(s) d s.
    \end{split}
  \end{equation*}
  This shows the claim about the weak derivative of $\mathcal{K}_0 u$. Note that by density and continuity $\frac{d}{d s} \mathcal{K}_0 u = \mathcal{K}_1 u$ also holds for $u \in L^2(\mathbb{R}; \mathbb{C})$.

  \textit{Step~3}.
  We show the claim of the lemma for $r=0$, i.e. that the mapping
  \begin{equation*}
    \mathcal{K}: L^2(\mathbb{R}; \mathbb{C}) \rightarrow H^1(\mathbb{R}; \mathbb{C}), \qquad \mathcal{K} u(s) := \int_\mathbb{R} k(s, t) u(t) d t,
  \end{equation*}
  is compact. First, note that $\mathcal{K}$ is well-defined by the results in \textit{Step 1 \& 2}. Next, consider the map 
  \begin{equation*}
    \mathcal{D}: H^1(\mathbb{R}; \mathbb{C}) \rightarrow L^2(\mathbb{R}; \mathbb{C}), \qquad \mathcal{D} f = f + f'.
  \end{equation*}
  With the help of the Fourier transform it is not difficult to see that $\mathcal{D}$ is bijective. Using the results from \textit{Step 1 \& 2} we find that $\mathcal{D} \mathcal{K} = \mathcal{K}_0 + \mathcal{K}_1$ and the right hand side defines a compact operator in $L^2(\mathbb{R}; \mathbb{C})$. Therefore, $\mathcal{K} = \mathcal{D}^{-1} (\mathcal{K}_0 + \mathcal{K}_1)$ is compact as an operator from $L^2(\mathbb{R}; \mathbb{C})$ to $H^1(\mathbb{R}; \mathbb{C})$.

  \textit{Step~4}.
  We claim that $\mathcal{K}$ admits  a compact extension $\widetilde{\mathcal{K}}:H^{-1}(\mathbb{R}; \mathbb{C})\rightarrow L^2(\mathbb{R}; \mathbb{C})$. Due to the assumptions on the function $\overline{k(t, s)}$ in~\eqref{assumption_k} it follows as in \textit{Step~1--3} that the map 
  \begin{equation*}
    \mathcal{L}: L^2(\mathbb{R}; \mathbb{C}) \rightarrow H^1(\mathbb{R}; \mathbb{C}), \qquad \mathcal{L} u(s) := \int_\mathbb{R} \overline{k(t, s)} u(t) d t,
  \end{equation*}
  is compact. Hence, also the anti-dual operator
  \begin{equation*}
    \widetilde{\mathcal{K}} := \mathcal{L}': H^{-1}(\mathbb{R}; \mathbb{C}) \rightarrow L^2(\mathbb{R}; \mathbb{C})
  \end{equation*}
  is compact. We claim that $\widetilde{\mathcal{K}}$ is an extension of $\mathcal{K}$. To see this, we note that $\mathcal{L} u = \mathcal{K}_0^* u$ holds for any $u \in L^2(\mathbb{R}; \mathbb{C})$, and compute for $u, v \in L^2(\mathbb{R}; \mathbb{C})$
  \begin{equation*}
    \begin{split}
      \big( \widetilde{\mathcal{K}} u, v \big)_{L^2(\mathbb{R}; \mathbb{C})} = \big( u, \mathcal{L} v \big)_{H^{-1}(\mathbb{R}; \mathbb{C}) \times H^{1}(\mathbb{R}; \mathbb{C})} &= \big( u, \mathcal{L} v \big)_{L^2(\mathbb{R}; \mathbb{C})}\\
      &= \big( \mathcal{K}_0 u, v \big)_{L^2(\mathbb{R}; \mathbb{C})} = \big( \mathcal{K} u, v \big)_{L^2(\mathbb{R}; \mathbb{C})};
    \end{split}
  \end{equation*}
  in the latter computation we used the symbol $( \cdot, \cdot )_{H^{-1}(\mathbb{R}; \mathbb{C}) \times H^{1}(\mathbb{R}; \mathbb{C})}$ for the sesquilinear duality product in $H^{-1}(\mathbb{R}; \mathbb{C}) \times H^{1}(\mathbb{R}; \mathbb{C})$.
  Since this is true for all $v \in L^2(\mathbb{R}; \mathbb{C})$, we conclude that $\widetilde{\mathcal{K}} u = \mathcal{K} u$, which shows that $\widetilde{\mathcal{K}}$ is a compact extension of $\mathcal{K}$.

  \textit{Step~5}.
  By the results in \textit{Step~3 \& 4} the claim of the lemma is true for $r=-1$ and $r=0$. By interpolation, one can apply, e.g., \cite[Theorem~10]{CW95} with $X_0 = H^{-1}(\mathbb{R}; \mathbb{C})$, $X_1  = L^2(\mathbb{R}; \mathbb{C})$, and $W = H^{-2}(\mathbb{R}; \mathbb{C})$, it follows that also for $r \in (-1, 0)$ the map $\mathcal{K}$ gives rise to a compact operator from $H^r(\mathbb{R}; \mathbb{C})$ to $H^{r+1}(\mathbb{R}; \mathbb{C})$.
\end{proof}

\begin{proof}[Proof of Proposition~\ref{proposition_C_z_Gamma}]
  Define the operator 
  \begin{equation*} 
    \mathcal{K} := U_{\gamma} V (\sigma \cdot \nu) \mathcal{C}_z^{m, c} (\sigma \cdot \nu) V^* U_\gamma^{-1} - U_x \sigma_2 \mathcal{C}_z^{m, c, \Sigma} \sigma_2 U_x^{-1}.
  \end{equation*}
  The map $\mathcal{K}$ is an integral operator and we will do a fine analysis of the integral kernel of $\mathcal{K}$, so that we can apply Lemma~\ref{lemma_int_op_compact} for $r = -\frac{1}{2}$ to each of the entries of the $2\times 2$ block operator matrix $\mathcal{K}$ 
  to conclude the claimed result.

  Recall that $\gamma$ is given by~\eqref{arc_length_Gamma_intro}. We use for $y = (y_1, y_2) \in \mathbb{R}^2$ the notation $\boldsymbol{y} = y_1 + i y_2$, implying that $|y|_{\mathbb{R}^2} = |\boldsymbol{y}|_\mathbb{C}$. 
  Furthermore, for $s \in \mathbb{R}$ one has $\boldsymbol{t}(\gamma(s)) = \dot{\boldsymbol{\gamma}}(s)$ and $\boldsymbol{\nu}(\gamma(s)) = \dot{\gamma}_2(s) - i \dot{\gamma_1}(s) = -i \dot{\boldsymbol{\gamma}}(s)$, and thus
  \begin{equation*}
    V(\gamma(s)) = \begin{pmatrix} 1 & 0 \\ 0 & \overline{\dot{\boldsymbol{\gamma}}(s)} \end{pmatrix} \quad \text{and} \quad \sigma \cdot \nu(\gamma(s)) = \begin{pmatrix} 0 & i \overline{\dot{\boldsymbol{\gamma}}(s)} \\ -i \dot{\boldsymbol{\gamma}}(s) & 0 \end{pmatrix}.
  \end{equation*}
  Therefore, by using the definition of $U_\gamma, U_x$, and $\mathcal{C}_z^{m, c}$ we see for $u \in L^2(\mathbb{R}; \mathbb{C}^2)$ that
  \begin{equation} \label{def_K_int_op} 
    \begin{split}
      \mathcal{K} u(s) 
      &= \lim_{\varepsilon \searrow 0} \int_{\mathbb{R} \setminus B(s, \varepsilon)} k(s, t) u (t) d t,
    \end{split}
  \end{equation}
  where
  \begin{equation*} 
    \begin{split}
      k(s, t) := \begin{pmatrix} 0 & i \overline{\dot{\boldsymbol{\gamma}}(s)} \\ -i & 0 \end{pmatrix} G_z^{m, c}\big(\gamma(s) - \gamma(t)\big) \begin{pmatrix} 0 & i \\ -i \dot{\boldsymbol{\gamma}}(t) & 0 \end{pmatrix} - \sigma_2 G_z^{m, c} (s - t, 0) \sigma_2.
    \end{split}
  \end{equation*}  
  Define the number $\zeta(z) := - i \sqrt{\tfrac{z^2}{c^2} - (m c)^2}$ and note that
  \begin{equation*}
    \begin{split}
      \frac{\sigma \cdot (\gamma(s) - \gamma(t))}{|\gamma(s) - \gamma(t)|} &= \frac{1}{|\gamma(s) - \gamma(t)|} \begin{pmatrix} 0 & \overline{\boldsymbol{\gamma}(s)} - \overline{\boldsymbol{\gamma}(t)} \\ \boldsymbol{\gamma}(s) - \boldsymbol{\gamma}(t) & 0\end{pmatrix} \\
      &= |\gamma(s) - \gamma(t)| \begin{pmatrix} 0 & \frac{1}{\boldsymbol{\gamma}(s) - \boldsymbol{\gamma}(t)} \\ \frac{1}{\overline{\boldsymbol{\gamma}(s)} - \overline{\boldsymbol{\gamma}(t)}} & 0\end{pmatrix}.
    \end{split}
  \end{equation*}
  Then, by the definition of $G_z^{m, c}$ in~\eqref{def_G_lambda}, the function $k$ can be further decomposed as $k(s, t) = k_1(s, t) + k_2(s, t) + k_3(s, t)$ with
  \begin{equation*}
    \begin{split}
      k_1(s, t) &:= \frac{1}{2 \pi c} \overline{\dot{\boldsymbol{\gamma}}(s)} \big(\dot{\boldsymbol{\gamma}}(t) - \dot{\boldsymbol{\gamma}}(s))   \begin{pmatrix} \tfrac{z}{c} - m c  & 0 \\ 0 & 0 \end{pmatrix} K_0 \left(\zeta(z) | \gamma(s) - \gamma(t) |\right),\\
      k_2(s, t) &:= \frac{1}{2 \pi c} \left( \frac{z}{c} \sigma_0 - m c \sigma_3\right) \big( K_0 \left(\zeta(z) | \gamma(s) - \gamma(t) |\right) - K_0 \left(\zeta(z) | s - t |\right) \big), \\
      k_3(s, t) &:= \frac{i \zeta(z)}{2 \pi c} \Bigg( K_1 \left( \zeta(z) | s - t |\right) |s - t| \begin{pmatrix} 0 & \frac{1}{s-t} \\ \frac{1}{s-t} & 0 \end{pmatrix} \\
      &\qquad  - K_1 \left( \zeta(z) | \gamma(s) - \gamma(t) |\right) | \gamma(s) - \gamma(t) | \Bigg( \begin{matrix} 0 & \tfrac{\overline{\dot{\boldsymbol{\gamma}}(s)}}{\overline{\boldsymbol{\gamma}(s)} - \overline{\boldsymbol{\gamma}(t)}} \\ \tfrac{\dot{\boldsymbol{\gamma}}(t)}{\boldsymbol{\gamma}(s) - \boldsymbol{\gamma}(t)} & 0\end{matrix} \Bigg) \Bigg).
    \end{split}
  \end{equation*}
  Denote the corresponding integral operators by $\mathcal{K}_j$, $j \in \{ 1, 2, 3 \}$, i.e. 
  \begin{equation*}
    \mathcal{K}_j \varphi(s) = \lim_{\varepsilon \searrow 0} \int_{\mathbb{R} \setminus B(s, \varepsilon)} k_j(s, t) \varphi (t) d t,\quad s \in \mathbb{R},~\varphi \in C_0^\infty(\mathbb{R}; \mathbb{C}^2).
  \end{equation*}
  We show that the kernels $k_j$ satisfy the assumptions in Lemma~\ref{lemma_int_op_compact} implying that each of these operators is not strongly singular (i.e. the limit for $\varepsilon \searrow 0$ can be removed) and has the claimed mapping properties.

  Let $M > 0$ be the number specified in the definition of $\gamma$ in~\eqref{arc_length_Gamma_intro}. First, we note that, due to the asymptotic properties of the modified Bessel functions and their derivatives for large arguments \cite[Sections 9.6 and 9.7]{AS84} and~\eqref{equation_bi_Lipschitz}, there exist constants $M_1, \kappa_1 > 0$ such that for $j \in \{ 1,2,3 \}$ and $p \in \{ 0, 1 \}$
  \begin{equation} \label{equation_k_1_2_large_arguments}
    \max \left\{ \left|\frac{\partial^{p} k_j(s, t)}{\partial s^p }\right|, \left|\frac{\partial^{p} \overline{k_j(t, s)}}{\partial s^p }\right| \right\} \leq M_1 e^{\kappa_1 |s - t|}, \quad |s-t| \geq M.
  \end{equation}
  
  Next, taking the power series representation of $K_0$ from \cite[equation~9.6.13]{AS84} into account, there exist a constant $M_2$ and holomorphic functions $g_1, g_2$ such that
  \begin{equation} \label{decomposition_K_0}
    K_0(\xi) = - \log \xi + M_2 + g_1(\xi^2) + \xi^2 g_2(\xi^2) \log \xi = - \log \xi + g_3(\xi), \quad \xi \in \mathbb{C} \setminus \{ 0 \},
  \end{equation}
  with the $C^1$-smooth function $g_3(\xi) := M_2 + g_1(\xi^2) + \xi^2 g_2(\xi^2) \log \xi$; cf. the proof of \cite[Lemma~3.2]{BHOP20} for a similar consideration. 
  We will often use that the function
  \begin{equation} \label{equation_log_smooth}
    \log\big( \zeta(z) |\gamma(s) - \gamma(t)| \big) - \log \big( \zeta(z) |s - t| \big) = \log\left| \frac{\boldsymbol{\gamma}(s) - \boldsymbol{\gamma}(t)}{s-t} \right|
  \end{equation}
  is $C^\infty$-smooth. For $s \neq t$, this follows from the injectivity and the smoothness of $\gamma$, and for $s=t$ from a Taylor series expansion of $\boldsymbol{\gamma}$ and~\eqref{equation_bi_Lipschitz}.

  In the following, the properties of $\mathcal{K}_1$ and its kernel $k_1$ are further analyzed.
  With~\eqref{decomposition_K_0} and~\eqref{equation_log_smooth} it is not difficult to see that $k_1$ can be written as
  \begin{equation*}
    \begin{split} 
      k_1(s, t) &= \frac{1}{2 \pi c} \overline{\dot{\boldsymbol{\gamma}}(s)} \big(\dot{\boldsymbol{\gamma}}(t) - \dot{\boldsymbol{\gamma}}(s)\big)   \begin{pmatrix} \tfrac{z}{c} - m c  & 0 \\ 0 & 0 \end{pmatrix} \\
      &\qquad \cdot \left( - \log \big(\zeta(z) | s - t |\big) - \log\left| \frac{\boldsymbol{\gamma}(s) - \boldsymbol{\gamma}(t)}{s-t} \right| + g_3 \big(\zeta(z) | \gamma(s) - \gamma(t) |\big) \right).
    \end{split}
  \end{equation*}  
  Since $\boldsymbol{\gamma}$ is smooth, the function $(\dot{\boldsymbol{\gamma}}(t) - \dot{\boldsymbol{\gamma}}(s)) \log (\zeta(z) | s - t |)$ is continuous in $\mathbb{R}^2$, differentiable for $s \neq t$, and its first derivatives are bounded by $ \widetilde{M}_1 \big|\log |s - t| \big| + \widetilde{M}_2$ for some constants $\widetilde{M}_1, \widetilde{M}_2 > 0$.
  Hence, $k_1$ is continuous in $\mathbb{R}^2$ and differentiable for $s \neq t$, for $p \in \{ 0, 1 \}$ one has 
  \begin{equation} \label{estimate_k_1} 
    \max \left\{ \left|\frac{\partial^{p} k_1(s, t)}{\partial s^p }\right|, \left|\frac{\partial^{p} \overline{k_1(t, s)}}{\partial s^p }\right| \right\} \leq M_3 \big|\log |s - t| \big| + M_4, \quad |s|, |t| \leq 2 M,
  \end{equation}
  where $M_3, M_4$ are positive constants, and $k_1(s, t) = 0$, if $s, t < -M$ or $s, t > M$, as in that case $\dot{\boldsymbol{\gamma}}(s) = \dot{\boldsymbol{\gamma}}(t)$. Since $k_1$ fulfils~\eqref{equation_k_1_2_large_arguments}, this function satisfies the assumptions in Lemma~\ref{lemma_int_op_compact} and thus, $\mathcal{K}_1: H^{-1/2}(\mathbb{R}; \mathbb{C}^2) \rightarrow H^{1/2}(\mathbb{R}; \mathbb{C}^2)$ is compact.
  
  Next, $\mathcal{K}_2$ and its kernel $k_2$ are considered.
  With~\eqref{decomposition_K_0} and~\eqref{equation_log_smooth} one finds that $k_2 \in C^1(\mathbb{R}^2; \mathbb{C}^2)$. In particular, this shows for a constant $M_5 > 0$ that
  \begin{equation*}
    \max \left\{ \left|\frac{\partial^{p} k_2(s, t)}{\partial s^p }\right|, \left|\frac{\partial^{p} \overline{k_2(t, s)}}{\partial s^p }\right| \right\} \leq M_5, \quad |s|, |t| \leq 2 M.
  \end{equation*}
  Finally, since for $s, t \geq M$ 
  \begin{equation} \label{equation_absolute_value_gamma}
    |\gamma(s) - \gamma(t)| = \left| \begin{pmatrix} a_+ \\  c_+ \end{pmatrix} (s-t) \right| = |s-t|
  \end{equation}
  holds, we get that $k_2(s, t) = 0$, if $s, t \geq M$. A similar consideration can be done also for $s, t \leq -M$. With~\eqref{equation_k_1_2_large_arguments} one sees that $k_2$ satisfies~\eqref{assumption_k} and hence, by Lemma~\ref{lemma_int_op_compact} the map $\mathcal{K}_2: H^{-1/2}(\mathbb{R}; \mathbb{C}^2) \rightarrow H^{1/2}(\mathbb{R}; \mathbb{C}^2)$ is compact. 
  
  It remains to analyze $\mathcal{K}_3$ and its kernel $k_3$. By the power series representation of $K_1$ in \cite[equation~9.6.11]{AS84} there exist holomorphic functions $g_4, g_5$ such that
  \begin{equation}  \label{eq:K_1_exp}
    K_1(\xi) = \frac{1}{\xi} + \xi g_4(\xi^2) \log \xi + \xi g_5(\xi^2), \quad \xi \in \mathbb{C} \setminus \{ 0 \};
  \end{equation}
  cf. the proof of \cite[Lemma~3.2]{BHOP20} for a similar consideration.
  Hence, we can further decompose $k_3(s, t)$ as $k_3(s, t) = k_4(s, t) + k_5(s, t) + k_6(s, t)$ with
  \begin{equation*}
    \begin{split}
      k_4(s, t) &:= \frac{i}{2 \pi c} \begin{pmatrix} 0 & \frac{1}{s-t} \\ \frac{1}{s-t} & 0 \end{pmatrix} - \frac{i}{2 \pi c} \begin{pmatrix} 0 & \frac{\overline{\dot{\boldsymbol{\gamma}}(s)}}{\overline{\boldsymbol{\gamma}(s)} - \overline{\boldsymbol{\gamma}(t)}} \\ \frac{\dot{\boldsymbol{\gamma}}(t)}{\boldsymbol{\gamma}(s) - \boldsymbol{\gamma}(t)} & 0\end{pmatrix}, \\
      k_5(s, t) &:= \frac{i \zeta(z)^2}{2 \pi c} \left( g_4(\zeta(z)^2 | s - t |^2) \log (\zeta(z) | s - t |) \begin{pmatrix} 0 & s-t \\ s-t & 0 \end{pmatrix}\right. \\
      &\qquad \qquad \qquad -  g_4(\zeta(z)^2 | \gamma(s) - \gamma(t) |^2) \log (\zeta(z) | \gamma(s) - \gamma(t) |) \\
      &\qquad \qquad\qquad \qquad \cdot \left. \begin{pmatrix} 0 & \overline{\dot{\boldsymbol{\gamma}}(s)} (\boldsymbol{\gamma}(s) - \boldsymbol{\gamma}(t)) \\ \dot{\boldsymbol{\gamma}}(t)(\overline{\boldsymbol{\gamma}(s)} - \overline{\boldsymbol{\gamma}(t)}) & 0\end{pmatrix} \right),\\
      k_6(s, t) &:= \frac{i \zeta(z)^2}{2 \pi c} \left( g_5(\zeta(z)^2 | s - t |^2) \begin{pmatrix} 0 & s-t \\ s-t & 0 \end{pmatrix}\right. \\
      &\qquad  - \left. g_5(\zeta(z)^2 | \gamma(s) - \gamma(t) |^2)  \begin{pmatrix} 0 & \overline{\dot{\boldsymbol{\gamma}}(s)} (\boldsymbol{\gamma}(s) - \boldsymbol{\gamma}(t)) \\ \dot{\boldsymbol{\gamma}}(t)(\overline{\boldsymbol{\gamma}(s)} - \overline{\boldsymbol{\gamma}(t)}) & 0\end{pmatrix} \right).
    \end{split}
  \end{equation*}
  Since $g_5$ is analytic, we observe that $k_6$ can be extended to a $C^\infty$-smooth function in $\mathbb{R}^2$. The same is also true for $k_4$. Indeed, employing the fundamental theorem of calculus twice, we get
  \begin{equation*}
    \begin{split}
      \frac{1}{s-t} - \frac{\dot{\boldsymbol{\gamma}}(t)}{\boldsymbol{\gamma}(s) - \boldsymbol{\gamma}(t)}  &= \frac{\boldsymbol{\gamma}(s) - \boldsymbol{\gamma}(t) - \dot{\boldsymbol{\gamma}}(t) (s-t)}{(s-t)(\boldsymbol{\gamma}(s) - \boldsymbol{\gamma}(t))} \\
      &= \frac{\int_0^1 \dot{\boldsymbol{\gamma}}(t + \xi (s-t)) (s-t) d \xi - \dot{\boldsymbol{\gamma}}(t) (s-t)}{(s-t)(\boldsymbol{\gamma}(s) - \boldsymbol{\gamma}(t))} \\
      &= \int_0^1 \int_0^1 \ddot{\boldsymbol{\gamma}}(t + \zeta \xi (s-t)) \xi  d \zeta d \xi \cdot \frac{s - t}{\boldsymbol{\gamma}(s) - \boldsymbol{\gamma}(t)}.
    \end{split}
  \end{equation*}
  The first factor in the last line is a smooth function everywhere in $\mathbb{R}^2$, and due to~\eqref{equation_bi_Lipschitz} the same is true also for the second factor. Therefore, $k_4$ can be extended to a $C^\infty$-function at $s=t$.
  
  Next, we analyze the properties of $k_5(s, t)$ for $s, t$ being close to each other. In a very similar way as in the consideration of $k_1$ one gets that $k_5$ is continuous everywhere in $\mathbb{R}^2$ and that there exist  constants $M_6, M_7>0$ such that for 
  $p \in \{ 0, 1 \}$
  \begin{equation*} 
    \max \left\{ \left|\frac{\partial^{p} k_5(s, t)}{\partial s^p }\right|, \left|\frac{\partial^{p} \overline{k_5(t, s)}}{\partial s^p }\right| \right\} \leq M_6 \big|\log|s-t|\big| +M_7, \quad |s|,|t| \leq 2 M,
  \end{equation*}
  holds; cf.~\eqref{estimate_k_1}. Since $k_4$ and $k_6$ are smooth functions and $k_3 = k_4 + k_5 + k_6$, we conclude that there exist  constants $M_8, M_9>0$ such that for 
  $p \in \{ 0, 1 \}$
  \begin{equation} \label{equation_k_3_small_arguments}
   \max \left\{ \left|\frac{\partial^{p} k_3(s, t)}{\partial s^p }\right|, \left|\frac{\partial^{p} \overline{k_3(t, s)}}{\partial s^p }\right| \right\} \leq M_8 \big|\log|s-t|\big| +M_9, \quad |s|, |t| \leq 2 M,
  \end{equation}
  is satisfied.
  
  Finally, we show that $k_3(s, t) = 0$ for $s, t > M$ or $s, t < -M$. Note that for $s, t \geq M$ we have~\eqref{equation_absolute_value_gamma}
  and 
\begin{equation*}
\frac{\dot{\boldsymbol{\gamma}}(t)}{\boldsymbol{\gamma}(s)-\boldsymbol{\gamma}(t)}=\frac{a_+ +ic_+}{a_+(s-t)+ic_+(s-t)}=\frac{1}{s-t}.
\end{equation*}
Similar identities are also true for $s, t \leq -M$. Hence, we conclude from the definition of $k_3$ that $k_3(s, t) = 0$ for $s, t > M$ or $s, t < -M$. Combining this with~\eqref{equation_k_1_2_large_arguments} and~\eqref{equation_k_3_small_arguments}, it follows from Lemma~\ref{lemma_int_op_compact} that also $\mathcal{K}_3: H^{-1/2}(\mathbb{R}; \mathbb{C}^2) \rightarrow H^{1/2}(\mathbb{R}; \mathbb{C}^2)$ is compact. Together with the mapping properties of $\mathcal{K}_1$ and $\mathcal{K}_2$ shown above this yields that 
\begin{equation*}
  \mathcal{K} = \mathcal{K}_1 + \mathcal{K}_2 + \mathcal{K}_3: H^{-1/2}(\mathbb{R}; \mathbb{C}^2) \rightarrow H^{1/2}(\mathbb{R}; \mathbb{C}^2)
\end{equation*}
is a compact operator. The proof of Proposition~\ref{proposition_C_z_Gamma} is complete.
\end{proof}
    
\end{appendix}



\end{document}